\documentclass[11pt]{article}
\usepackage{amsfonts}
\usepackage{latexsym}
\usepackage{amsmath}
\usepackage{amssymb}
\usepackage{amsthm}
\usepackage{amsmath}
\usepackage{verbatim}
\usepackage{hyperref}
\usepackage{bibspacing}
 \usepackage{tikz}

\newtheorem{thm}{Theorem}[section]
\newtheorem*{thm*}{Theorem}

\newtheorem{corollary}[thm]{Corollary}

\newtheorem{lemma}[thm]{Lemma}
\newtheorem{prop}[thm]{Proposition}
\newtheorem*{prop*}{Proposition}
\newtheorem{proposition}[thm]{Proposition}
\newtheorem{conj}[thm]{Conjecture}
\newtheorem*{conj*}{Conjecture}
\newtheorem{defn}[thm]{Definition}
\newtheorem*{dfn*}{Definition}
\theoremstyle{definition}
\newtheorem{rem}[thm]{\textbf{Remark}}
\newtheorem*{rmk*}{Remark}
\newtheorem*{fact*}{Fact}

\theoremstyle{proof}

\newcommand{\conv}{\text{conv}}

\newcommand{\sspan}{\text{span}}

\newcommand{\norm}[1]{\left\Vert#1\right\Vert}

\newcommand{\abs}[1]{\left\vert#1\right\vert}
\newcommand{\set}[1]{\left\{#1\right\}}
\newcommand{\brac}[1]{\left(#1\right)}
\newcommand{\scalar}[1]{\left \langle #1 \right \rangle}
\newcommand{\sscalar}[1]{\langle #1 \rangle}
\newcommand{\Real}{\mathbb{R}}
\newcommand{\R}{\mathbb{R}}
\newcommand{\RP}{\mathbb{RP}}

\renewcommand{\S}{\mathbb{S}}
\newcommand{\I}{\mathcal{I}}
\newcommand{\II}{\text{II}}

\renewcommand{\H}{\mathcal{H}}
\newcommand{\eps}{\epsilon}
\renewcommand{\L}{\mathcal{L}}
\newcommand{\K}{\mathcal{K}}
\newcommand{\Q}{\mathcal{Q}}
\newcommand{\GG}{\mathcal{G}}
\DeclareMathOperator{\Img}{\mathop{Im}}
\DeclareMathOperator{\Ker}{\mathop{Ker}}
\DeclareMathOperator{\interior}{\mathop{int}}
\newcommand{\SO}{{\rm{SO}}}

\numberwithin{equation}{section}

\begin{document}

\renewcommand*{\thefootnote}{\fnsymbol{footnote}}

\author{Emanuel Milman\textsuperscript{$*$,$\dagger$}
\and
Amir Yehudayoff\textsuperscript{$*$,$\ddagger$}
}

\footnotetext[1]{Department of Mathematics, Technion-Israel Institute of Technology, Haifa 32000, Israel.}
\footnotetext[2]{Email: emilman@tx.technion.ac.il.}
\footnotetext[3]{Email: amir.yehudayoff@gmail.com. }

\begingroup    \renewcommand{\thefootnote}{}    
    \footnotetext{The research leading to these results is part of a project that has received funding from the European Research Council (ERC) under the European Union's Horizon 2020 research and innovation programme (grant agreement No 637851) and from the Israeli Science Foundation grant \#1162/15.}
\endgroup

\title{Sharp Isoperimetric Inequalities\\ for Affine Quermassintegrals}

\date{\nonumber} 
\maketitle

\begin{abstract}
The affine quermassintegrals associated to a convex body in $\R^n$ are affine-invariant analogues of the classical intrinsic volumes from the Brunn--Minkowski theory,
and thus constitute a central pillar of Affine Convex Geometry. 
They were introduced in the 1980's by E.~Lutwak, who conjectured that among all convex bodies of a given volume, the $k$-th affine quermassintegral is minimized precisely on the family of ellipsoids. The known cases $k=1$ and $k=n-1$ correspond to the classical Blaschke--Santal\'o and Petty projection inequalities, respectively. In this work we confirm Lutwak's conjecture, including characterization of the equality cases, for all values of $k=1,\ldots,n-1$, in a single unified framework. In fact, it turns out that ellipsoids are the only \emph{local} minimizers with respect to the Hausdorff topology. 

\smallskip

For the proof, we introduce a number of new ingredients, including a novel construction of the Projection Rolodex of a convex body. 
In particular, from this new view point, Petty's inequality is interpreted as an integrated form of a generalized Blaschke--Santal\'o inequality for a new family of polar bodies encoded by the Projection Rolodex. We extend these results to more general $L^p$-moment quermassintegrals, and interpret the case $p=0$ as a sharp averaged Loomis--Whitney isoperimetric inequality.

\end{abstract}

\newpage

\section{Introduction}

Let $K$ denote a convex compact set with non-empty interior (``convex body") in Euclidean space $\R^n$. 
A classical observation attributed to J.\,Steiner (see e.g. \cite[(4.8) and p. 223]{Schneider-Book-2ndEd})
is that the volume of the outer parallel body $K + t B_2^n$ is an $n$-degree polynomial in $t > 0$:
\begin{equation} \label{eq:intro-W1}
| K + t B_2^n| = \sum_{k=0}^n  {n \choose k} W_k(K) t^{n-k} . 
\end{equation}
Here and throughout, $B_2^n$ denotes the Euclidean unit ball in $\R^n$, $|\cdot|$ denotes Lebesgue measure on the corresponding linear space, and $A + B := \{ a + b \; ; \; a \in A , b \in B \}$ denotes setwise (Minkowski) addition.  Note our convention of labeling the coefficient of $t^{n-k}$ by $W_k(K)$ instead of the more traditional $W_{n-k}(K)$, as this reflects the order of homogeneity of $W_k(K)$ under scaling of $K$.

The coefficients $\{W_k(K)\}_{k=0,\ldots,n}$ were dubbed ``quermassintegrals" (``cross-sectional measures", also known as ``intrinsic volumes") by H.\,Minkowski, who extended Steiner's observation to general convex summands of the form $\sum_{i=1}^m t_i K_i$, leading to the notion of mixed-volume and the birth of the Brunn--Minkowski theory \cite[Chapter 4]{BuragoZalgallerBook}, \cite[Chapters 5-7]{Schneider-Book-2ndEd}. The quermassintegrals and their linear combinations were subsequently characterized by H.\,Hadwiger as the unique continuous valuations on convex compacts subsets of $\R^n$ which are invariant under rigid motions \cite[6.1.10]{Hadwiger-Book}, \cite[Theorem 6.4.14]{Schneider-Book-2ndEd}. 
Consequently, they carry important geometric information on $K$; for example, it is clear from (\ref{eq:intro-W1}) that $W_n(K) = |K|$, $W_0(K) = |B_2^n|\chi(K)$ where $\chi$ is the Euler characteristic, and that $W_{n-1}(K) = \frac{1}{n} S(K)$ where $S(K)$ denotes the surface-area of $K$.

There are numerous additional ways of equivalently defining $W_k(K)$, spanning from differential-geometric to integral-geometric \cite[Section 19.3]{BuragoZalgallerBook}. Of particular relevance is Kubota's formula \cite[p. 301]{Schneider-Book-2ndEd}, \cite[p. 222]{SchneiderWeil-Book}, stating that for $k=1,\ldots,n$:
\begin{equation} \label{eq:intro-W2}
W_{k}(K) = \frac{|B_2^n|}{|B_2^k|} \int_{G_{n,k}} |P_F K| \sigma_{n,k}(dF)  ,
\end{equation}
where $G_{n,k}$ denotes the Grassmannian of all $k$-dimensional linear subspaces of $\R^n$ endowed with its unique Haar probability measure $\sigma_{n,k}$, and $P_F$ denotes orthogonal projection onto $F \in G_{n,k}$. For example, it follows from (\ref{eq:intro-W2}) that $W_1(K)$ is proportional to the mean-width of $K$.

The following sharp isoperimetric inequalities for the quermassintegrals are well-known and classical. Here and throughout this work, $B_K$ denotes the (centered) Euclidean ball having the same volume as $K$. 
\begin{thm}[Favard, Alexandrov--Fenchel] \label{thm:FAF}
For any convex body $K \subset \R^n$ and $k=1,\ldots,n-1$: 
\[
 W_k(K) \geq W_k(B_K),
\]
 with equality for a given $k$ if and only if $K$ is a Euclidean ball. 
\end{thm}
The above inequality is a simple consequence of the celebrated Alexandrov--Fenchel inequality for the mixed-volumes \cite[Theorem 7.3.1]{Schneider-Book-2ndEd}, established by A.\,D.\,Alexandrov \cite{Alexandrov-II,Alexandrov-IV} and independently sketched by W.\,Fenchel \cite{Fenchel1936} in 1936, and in fact also follows from a simpler prior inequality due to J.\,Favard \cite{Favard1933}. The equality case above is presumably due to Alexandrov; we refer to \cite[pp. 398-399]{Schneider-Book-2ndEd} and \cite[Sections 20.2-20.3]{BuragoZalgallerBook} for a historical discussion, and to \cite[Subsection 6.4.4, Satz XI]{Hadwiger-Book}, \cite[(9)]{Shephard-ShadowSystems} for alternative approaches. In particular, the case $k=n-1$ recovers the sharp isoperimetric inequality for the surface-area, and the case $k=1$ recovers the sharp isoperimetric inequality for the mean-width due to P.\,Urysohn \cite{Urysohn-Inq} (cf.~\cite[(7.21)]{Schneider-Book-2ndEd}).

\subsection{Affine Quermassintegrals}

The quermassintegrals $W_k(K)$ are not invariant under volume-preserving affine transformations of $K$ (for $k=1,\ldots,n-1$), and so it is tempting to find an analogous notion which \emph{is} invariant under such transformations (we will simply say ``affine-invariant"). Given $k=1,\ldots,n$, the $k$-th \emph{affine quermassintegral} of $K$ was defined by E.~Lutwak in \cite{Lutwak-Isepiphanic} by replacing the $L^{1}$-norm in (\ref{eq:intro-W2}) by the $L^{-n}$-norm:
\begin{equation} \label{eq:def-affine-quermass}
\Phi_{k}(K) := \frac{|B_2^n|}{|B_2^k|} \brac{\int_{G_{n,k}} |P_F K|^{-n} \sigma_{n,k}(dF)}^{-\frac{1}{n}}  . 
\end{equation}
It was shown by Grinberg \cite{Grinberg-Affine-Invariant} that $K \mapsto \Phi_k(K)$ is indeed affine-invariant.
Consequently, the affine quermassintegrals have become a central pillar of Affine Convex Geometry.

\medskip

One of the most important problems in Affine Convex Geometry is to obtain sharp lower and upper bounds on $\Phi_k(K)$  and to characterize the extremizers (among all convex $K$ of a given volume). 
In \cite{Lutwak-Harmonic}, Lutwak put forth the following conjecture, whose higher-rank cases have become a major open problem in the field
 (see e.g. \cite[Open Problem 12.3]{Lutwak-Selected}, \cite[Problem 9.3]{GardnerGeometricTomography2ndEd}, \cite[(9.57)]{Schneider-Book-2ndEd}):
\begin{conj}[Lutwak] \label{conj:Lutwak} 
For any convex body $K \subset \R^n$ and $k=1,\ldots,n-1$:
\begin{equation} \label{eq:main}
\Phi_k(K) \geq \Phi_k(B_K) ,
\end{equation}
with equality for a given $k$ if and only if $K$ is an ellipsoid. 
\end{conj}
\noindent
By Jensen's inequality, as $W_k(K) \geq \Phi_k(K)$ and $W_k(B_K) = \Phi_k(B_K)$, 
the above would constitute a sharp, affine-invariant, strengthening of the classical isoperimetric inequalities of Theorem \ref{thm:FAF}. Furthermore, such a strengthening would be optimal, in the sense that the $L^{-n}$-norm in the definition (\ref{eq:def-affine-quermass}) and inequality (\ref{eq:main}) cannot be replaced by the $L^p$-norm for any $p < -n$ -- see Subsection \ref{subsec:intro-LW}.

\medskip

The only sharp results known thus far have been for the rank-one cases $k=1$ and $k=n-1$, which (as observed by Lutwak) turn out to be completely classical; we refer to Section~\ref{sec:notation} for missing standard definitions.  
For origin-symmetric convex bodies, $\Phi^{-n}_1(K)$ is proportional to the volume of the polar body $K^{\circ}$, and so the case $k=1$ of (\ref{eq:main}) amounts to the Blaschke--Santal\'o inequality:
\begin{equation} \label{eq:BS}
|K| |K^{\circ}| \leq |B_2^n|^2 .
\end{equation}
For general convex bodies, it is easy to check that (\ref{eq:main}) for $k=1$ is weaker than the Blaschke--Santal\'o inequality, stating that (\ref{eq:BS}) holds when $K$ is first centered at its Santal\'o point -- in that case, the corresponding polar body is denoted by $K^{\circ,s}$. The Blaschke--Santal\'o inequality was established by Blaschke \cite[p. 198]{Blaschke-Book} for $n\leq 3$
and Santal\'o \cite{Santalo-BS-paper} for general $n$ by using the isoperimetric inequality for affine surface area (see \cite[(10.28) and Section 10.7]{Schneider-Book-2ndEd}). The characterization of ellipsoids as the only cases of equality was established by Blaschke and Santal\'o under certain regularity assumptions on $K$, which were removed when $K=-K$ by Saint-Raymond \cite{SaintRaymond-Santalo}, who also gave a simplified proof of the origin-symmetric case. For general convex bodies without any regularity assumptions, the equality conditions of the Blaschke--Santal\'o inequality (and hence of (\ref{eq:main}) for $k=1$) were established by Petty \cite{Petty-AffineIsoperimetricProblems}.
Subsequent additional proofs (most of which include a characterization of equality) were obtained by Meyer--Pajor \cite{MeyerPajor-Santalo}, Lutwak--Zhang \cite{LutwakZhang-IntroduceLqCentroidBodies}, Lutwak--Yang--Zhang \cite{LYZ-MomentEntropyInqs}, Campi--Gronchi \cite{CampiGronchi-VolumeProductInqs} and Meyer--Reisner \cite{MeyerReisner-SantaloViaShadowSystems}, to name a few. 

\medskip

On the other extreme, $\Phi^{-n}_{n-1}(K)$ is proportional to the volume of the polar projection body $\Pi^* K = (\Pi K)^{\circ}$. Here $\Pi K$ denotes the projection body of $K$, defined (and shown to exist) by Minkowski as the convex body whose support function satisfies $h_{\Pi K}(\theta) = |P_{\theta^{\perp}} K|$. Consequently, the case $k = n-1$ of (\ref{eq:main}) amounts to Petty's projection inequality \cite{Petty-IsoperimetricProblems}:
\[
|\Pi^* K| \leq |\Pi^* B_K | ,
\]
with equality if and only if $K$ is an ellipsoid. Petty derived this result from the Busemann--Petty centroid inequality (see \cite[Chapter 9]{GardnerGeometricTomography2ndEd}), and the two inequalities are in fact equivalent to each other \cite{Lutwak-SomeAffine,Lutwak-Selected} (see also \cite[p. 572]{Schneider-Book-2ndEd}). Subsequent additional proofs of Petty's projection inequality with characterization of equality and extensions thereof were obtained by Lutwak--Yang--Zhang \cite{LYZ-Lp-PettyProjection,LYZ-OrliczProjectionBodies} and Campi--Gronchi \cite{CampiGronchi-LpBusemannPettyCentroid} (in fact, for the more general $L^p$ and Orlicz projection bodies). In \cite{Zhang-AffineSobolevInq}, Zhang extended Petty's projection inequality to more general compact sets, and showed that this more general version is equivalent to a sharp affine-invariant Sobolev inequality which is stronger than the classical (non-affine) one (and hence stronger than the classical isoperimetric inequality).  

\medskip

These classical cases $k \in \{1 , n-1\}$ are fundamental tools in Affine Convex Geometry, and have found further applications in
Asymptotic Geometric Analysis, Functional Inequalities and Concentration of Measure, Partial Differential Equations, Functional Analysis, the Geometry of Numbers, Discrete Geometry and Polytopal Approximations, Stereology and Stochastic Geometry, and Minkowskian Geometry (see \cite{Lutwak-Selected,AGA-Book-I} and the references therein).

\medskip

In this work, we confirm the remaining higher-rank cases $k=2,\ldots,n-2$ of Conjecture \ref{conj:Lutwak}, which constitute the main result of this work:
\begin{thm} \label{thm:main}
Conjecture \ref{conj:Lutwak} is true for all $k=1,\ldots,n-1$. 
\end{thm}

\medskip

Before describing several further extensions and our method of proof, 
let us briefly mention some additional related results. By employing methods from Asymptotic Geometric Analysis, Paouris and Pivovarov \cite{PaourisPivovarov-SmallBall} (see also \cite{DafnisPaouris-AffineQuermassintegrals,DannPaourisPivovarov-FlagManifolds})
have previously confirmed the inequality (\ref{eq:main}) up to a factor of $c^k$ for some constant $c > 0$. Zou and Xiong \cite{ZouXiong} have shown that $\Phi_k(K)$ is lower-bounded (up to normalization) by the volume of the $k$-th projection mean ellipsoid of $K$, which is however incomparable to the volume of $K$. 
Lutwak also proposed the related notion of \emph{dual} affine quermassintegrals $\tilde \Phi_k(K)$,
in which one replaces projections by sections and the $L^{-n}$-norm by the $L^{+n}$ one in the definition of $\Phi_k$. The analogous problem of obtaining a sharp upper bound on  
$\tilde \Phi_k(K)$ for $k=2,\ldots,n-1$ was resolved for convex bodies by Busemann--Straus \cite[p. 70]{BusemannStraus} and Grinberg \cite{Grinberg-Affine-Invariant}, who also characterized centered ellipsoids as the only cases of equality (see also \cite[Section 8.6]{SchneiderWeil-Book}). These results for $\tilde \Phi_k(K)$ were extended to arbitrary bounded Borel sets $K$ (for which the characterization of equality is much more delicate) by Gardner \cite{Gardner-DualAffineQuermassintegrals}.
Finally, the question of obtaining sharp \emph{upper} bounds on $\Phi_k(K)$ has a long history and deserves a survey in itself. Let us only mention that a sharp upper bound on $\Phi_1(K)$ amounts to Mahler's conjecture \cite{Mahler-Conjecture} (see also \cite[Section 12.1]{Lutwak-Selected}),
stating that the volume product $|K| |K^{\circ,s}|$ for general convex bodies $K$ is minimized on simplices, and on cubes for origin-symmetric $K$. This has been confirmed by Mahler \cite{Mahler-ProofIn2D} in $\R^2$, by Iriyeh and Shibata \cite{IriyehShibata-SymmetricMahlerIn3D} (see also \cite{MahlerIn3D-Simplified}) for origin-symmetric $K$ in $\R^3$, and up to a factor of $c^n$ by J.~Bourgain and V.~Milman \cite{BourgainMilman-ReverseSantalo} (see also \cite{Kuperberg-Mahler, Nazarov-Mahler,GPV-ReverseSantalo}). A sharp upper bound on $\Phi_{n-1}(K)$ (i.e. reverse Petty projection inequality) with characterization of simplices as the only cases of equality was obtained by Zhang \cite{Zhang-ReversePetty}. To the best of our knowledge, sharp upper bounds on $\Phi_k(K)$ for $1 \leq k \leq n-2$ and $n \geq 4$ remain wide open; some asymptotic non-sharp estimates may be found in \cite{DafnisPaouris-AffineQuermassintegrals,PaourisPivovarov-SmallBall,DannPaourisPivovarov-FlagManifolds}, see also \cite{Greeks-AffineQuermassForRandomPolytopes}.  We refer to the excellent monographs \cite{GardnerGeometricTomography2ndEd,Schneider-Book-2ndEd, SchneiderWeil-Book} and survey paper \cite{Lutwak-Selected} for additional exposition and context.

\subsection{Two extensions}

Our proof of Theorem \ref{thm:main} proceeds by using the classical tool of Steiner symmetrization. 
Let $S_u K$ denote the Steiner symmetral of $K$ in a given direction $u \in \S^{n-1}$.
We obtain the following \emph{stronger} version of both the inequality and equality cases of Conjecture \ref{conj:Lutwak}: 
\begin{thm} \label{thm:intro-Steiner} 
For any convex body $K \subset \R^n$, $k=1,\ldots,n-1$ and $u \in \S^{n-1}$:
\begin{equation} \label{eq:intro-Steiner}
\Phi_k(K) \geq \Phi_k(S_u K) ,
\end{equation}
with equality for a given $k$ and all $u \in \S^{n-1}$ if and only if $K$ is an ellipsoid. 
\end{thm}

The inequality (\ref{eq:intro-Steiner}) for $k=1$ when $K=-K$ is origin-symmetric was obtained by Meyer--Pajor \cite{MeyerPajor-Santalo} in their proof of the Blaschke--Santal\'o inequality (see also Lutwak--Zhang \cite{LutwakZhang-IntroduceLqCentroidBodies} and Campi--Gronchi \cite{CampiGronchi-VolumeProductInqs}).
 For $k=n-1$, (\ref{eq:intro-Steiner}) was shown by Lutwak--Yang--Zhang \cite{LYZ-Lp-PettyProjection,LYZ-OrliczProjectionBodies} in their proof of Petty's projection inequality (for the more general $L^p$ and Orlicz projection bodies). 
The cases $k=2,\ldots,n-2$ of (\ref{eq:intro-Steiner}) are new. 

\medskip

Surprisingly,  the equality case of Theorem \ref{thm:intro-Steiner} was, to the best of our knowledge, previously only known in the case $k=1$: for origin-symmetric convex $K=-K$ this is due to Meyer--Pajor \cite{MeyerPajor-Santalo} (see also Lutwak--Zhang \cite{LutwakZhang-IntroduceLqCentroidBodies}); Meyer--Reisner \cite{MeyerReisner-SantaloViaShadowSystems} prove an analogous result for the Blaschke--Santal\'o inequality for general convex bodies. Even in the classical case $k=n-1$ corresponding to Petty's projection inequality, the equality case of Theorem \ref{thm:intro-Steiner} appears to be new; note that for $L^p$-projection inequalities with $1 < p <\infty$ and more general strictly convex Orlicz functions, an analogous result was obtained by  Lutwak--Yang--Zhang \cite{LYZ-Lp-PettyProjection,LYZ-OrliczProjectionBodies}, but their equality analysis breaks down in the classical case $p=1$. This is consistent with our own analysis in this work, where the cases of equality when $1 \leq k \leq n-2$, while requiring several new ideas, are relatively simpler to establish, but the case $k=n-1$ involves a fair amount of additional work.
It is worthwhile to note that our approach avoids any regularity issues in both the proof of the inequality and in the analysis of equality cases, in contrast to some other approaches in the classical cases $k=1$ and $k=n-1$.

\medskip

A different (yet very related) strengthening of Theorem \ref{thm:main} is given by:
\begin{thm} \label{thm:intro-local}
For all $k=1,\ldots,n-1$, among all convex bodies in $\R^n$ of a given volume, ellipsoids are the only {\tt local} minimizers of $\Phi_k$ with respect to the Hausdorff topology. 
\end{thm}
For $k=1$ this was recently shown by Meyer--Reisner \cite{MeyerReisner-LocalSantalo} (in fact, they show that an analogous statement holds for the volume of $K^{\circ,s}$, yielding a slightly stronger result than the one above in the case of non-origin-symmetric convex bodies). The cases $k=2,\ldots,n-2$ including the classical case $k=n-1$ for the volume of the polar projection body are new. 

\subsection{The challenge}

All proofs of the rank-one classical cases $k \in \{1 , n-1\}$ commence by associating to $K$ a (convex) body $L_k(K)$ in $\R^n$ which encodes the function $G_{n,k} \ni F \mapsto |P_F K|$. In these cases, this is easy to do: by naturally identifying $G_{n,k}$ with real projective space $\RP^{n-1}$, extending the function homogeneously to $\R^n$, and considering its level set, one obtains the polar body ($k=1$) and polar projection body ($k=n-1$) of $K$. In particular, the volume of $L_k(K)$ coincides with $\Phi^{-n}_k(K)$, and the fact that $L_k(K)$ resides in a linear space makes it convenient for checking the effect of Steiner symmetrization of $K$ on $|L_k(K)|$. 
For other values of $k$, it is not at all clear what is the right body $L_k(K)$ to associate with the function $G_{n,k} \ni F \mapsto |P_F K|$, and more importantly, in which space it should reside, as the standard ways of mapping a linear space (such as $(\R^n)^k$) onto the cone over $G_{n,k}$ are either highly non-linear or highly non-injective. 

\medskip
Our proof utilizes a new body which we introduce, called the \emph{Projection Rolodex} of $K$. It does not reside in a linear space, but rather (as perhaps its name suggests) in a vector bundle over a lower-dimensional Grassmannian. Another crucial difference with the classical cases, where the body $L_k(K)$ 
depends on $K$ alone, is that the construction of the Projection Rolodex $L_{k,u}(K)$ also depends on the direction $u \in \S^{n-1}$ in which we perform the Steiner symmetrization; this is in stark contrast to the rank-one cases, where a single global body $L_k(K)$ may be used for all directions simultaneously. 
It turns out that it is not the usual Haar measure of $L_{k,u}(K)$ which is related to $\Phi^{-n}_k(K)$, but rather some auxiliary measure $\mu_{u}$, also depending on $u$, which we introduce. We thus replace the order of quantifiers compared to the classical proofs: we first select a direction $u$,  only then define the Projection Rolodex $L_{k,u}(K)$, and now our task is to verify that $\mu_{k,u}(L_{k,u}(K)) \leq \mu_{k,u}(L_{k,u}(S_u K))$. Having aligned everything with the symmetrization direction $u$, the remaining challenge is then to analyze how Steiner symmetrization affects $|P_F K|$ for $F \in G_{n,k}$, and so we embark on a systematic study of the latter in Section \ref{sec:main} (in fact, for general shadow systems). 

\medskip
The above scheme allows us to prove Theorems \ref{thm:main}, \ref{thm:intro-Steiner} and \ref{thm:intro-local} simultaneously for all values of $k$ in a single unified framework, revealing a surprising connection between the Blaschke--Santal\'o inequality and Petty's projection inequality. From this point of view, Petty's inequality may be interpreted as an integrated form of a generalized Blaschke--Santal\'o inequality for a new family of polar-bodies associated with a given convex body $K$, encoded by the Projection Rolodex. We do not know whether the Blaschke--Santal\'o inequality may dually be interpreted as a generalized Petty projection inequality. However, in Subsection \ref{subsec:Petty} we obtain a new extremely simple proof of Petty's projection inequality, which reveals a deeper duality with the Blaschke--Santal\'o inequality.

\subsection{$L^p$-moment quermassintegrals and averaged Loomis--Whitney}  \label{subsec:intro-LW}

An analogous statement to that of Conjecture \ref{conj:Lutwak} holds for the $L^p$-moment quermassintegrals $\Q_{k,p}$, replacing the $L^{-n}$-norm by the $L^p$-norm in the definition (\ref{eq:def-affine-quermass}): 
\begin{equation} \label{eq:intro-Q}
\Q_{k,p}(K) \geq \Q_{k,p}(B_K) \;\;\; \forall p \geq -n ,
\end{equation}
with equality for $p > -n$ if and only if $K$ is a Euclidean ball -- see Definition \ref{def:Q} and Theorem \ref{thm:Q}. For $p=1$ these are nothing but the classical isoperimetric inequalities of Theorem \ref{thm:FAF} for the quermassintegrals $W_k(K)$, and for $p=-1$ the corresponding isoperimetric inequalities for the harmonic quermassintegrals $\hat W_k(K)$ were established by Lutwak \cite{Lutwak-Harmonic} by bootstrapping Petty's projection inequality. It is possible to extend this bootstrapping all the way down to the value $p=-(k+1)$, see Remark \ref{rem:Petty-stuck}. Of course, Jensen's inequality implies that the family of inequalities (\ref{eq:intro-Q}) becomes stronger as $p$ decreases, and so our result for $p=-n$ is stronger than all of the above. It is easy to check that the value $p=-n$ is best possible, i.e. that (\ref{eq:intro-Q}) is simply false for $p < -n$, see Remark~\ref{rem:optimal-p}. Going below $p=-(k+1)$ all the way down to the optimal value $p=-n$ requires several new ideas when $2 \leq k \leq n-2$, as outlined above. 

\medskip
It is worthwhile to note that the case $p=0$ is of special interest, as (\ref{eq:intro-Q}) may then be interpreted as a sharp averaged Loomis--Whitney isoperimetric inequality. The classical Loomis--Whitney inequality \cite{LoomisWhitney} lower bounds the geometric average of all $k$-dimensional projections of a compact set $K$ onto the principle axes in terms of the volume of $K$, yielding a sharp result for the cube (aligned with the axes). As an application, Loomis and Whitney deduce an isoperimetric inequality for the surface area of $K$, but with non-sharp constant. This is expected, as their inequality depends on the choice of coordinate system. The case $p=0$ of (\ref{eq:intro-Q}) implies that if one chooses the coordinate system at random and takes the geometric average of all $k$-dimensional projections, an improvement over the original Loomis--Whitney inequality is possible (for convex $K$). Moreover, this improvement is sharp for the Euclidean ball and thus yields the \emph{sharp} constant in the classical isoperimetric inequality for the surface area -- see Subsection~\ref{subsec:LW}.

\subsection{Alexandrov--Fenchel-type inequalities} 

It is convenient to introduce:
\[
\I_{k,p}(K) := \frac{\Q_{k,p}(K)}{\Q_{k,p}(B_K)} = \brac{\frac{\int_{G_{n,k}} |P_F K|^{p} \sigma_{n,k}(dF)}{\int_{G_{n,k}} |P_F B_K|^{p} \sigma_{n,k}(dF)}}^{\frac{1}{p}} . 
\]
Note that $\I_{k,p}(B) = 1$ for any Euclidean ball $B$ and all $k,p$, that $\I_{n,p}(K) = 1$ for all $p$, and that (\ref{eq:intro-Q}) translates to $\I_{k,p}(K) \geq 1$ for all $p \geq -n$.

In the classical case $p=1$, Alexandrov's inequalities \cite[(7.67)]{Schneider-Book-2ndEd}, \cite[Subsection 20.2]{BuragoZalgallerBook} (a particular case of the Alexandrov--Fenchel inequalities) assert that:
\[
\I_{1,1}(K) \geq \I^{1/2}_{2,1}(K) \geq \ldots \geq \I^{1/k}_{k,1}(K) \geq \ldots \geq \I^{1/(n-1)}_{n-1,1}(K) \geq \I^{1/n}_{n,1}(K) = 1 .
\]
The following was proved by Lutwak for $p=-1$ and conjectured to hold for $p=-n$ in \cite{Lutwak-Harmonic} 
(see also \cite[Problem 9.5]{GardnerGeometricTomography2ndEd}):

\begin{conj} \label{conj:AF}
For all $p \in [-n,0]$ and for any convex body $K$ in $\R^n$:
\[
\I_{1,p}(K) \geq \I^{1/2}_{2,p}(K) \geq \ldots \geq \I^{1/k}_{k,p}(K) \geq \ldots \geq \I^{1/(n-1)}_{n-1,p}(K) \geq \I^{1/n}_{n,p}(K) = 1 .
\]
\end{conj}

Our isoperimetric inequality (\ref{eq:intro-Q}) establishes the inequality between each of the terms and the last one. In the next theorem, which builds upon Theorem \ref{thm:main}, we confirm ``half" of the above conjecture. 

\begin{thm} \label{thm:intro-half-conj}
For every $p \in [-n,0]$ and $1 \leq k \leq m \leq n$:
\[
\I^{1/k}_{k,p}(K) \geq \I^{1/m}_{m,p}(K) ,
\]
for any convex body $K$ in $\R^n$ whenever $m \geq -p$. When $k < m < n$, equality holds for $p \geq -m$ if and only if $K$ is a Euclidean ball. When $k < m = n$, equality holds for $p > -n$ ($p=-n$) if and only if $K$ is a Euclidean ball (ellipsoid). 
\end{thm}
In particular, this confirms the conjecture for all $p \in [-2,0]$, recovering the case $p=-1$ established by Lutwak in \cite{Lutwak-Harmonic}. The analogous statement for the dual $L^p$-moment quermassintegrals $\tilde \I_{k,p}$ for arbitrary bounded Borel sets (replacing $p$ by $-p$ and with the direction of the inequality reversed) was established by Gardner \cite[Theorem 7.4]{Gardner-DualAffineQuermassintegrals} in exactly the same corresponding range of parameters -- see Subsection \ref{subsec:AF}.  
Establishing Conjecture \ref{conj:AF} in the remaining half range $1 \leq k < m < -p$ is a fascinating problem. Lutwak's original conjecture (the case $p=-n$ above) is presently only established in the plane and for $m=n$ (and all $k$) by Theorem \ref{thm:intro-half-conj}. 

\subsection*{Organization} 

The rest of this work is organized as follows. In Section \ref{sec:notation} we introduce some standard notation. In Section \ref{sec:main} we provide a proof of the sharp inequalities (\ref{eq:main}) and (\ref{eq:intro-Steiner}) of Theorems \ref{thm:main} and \ref{thm:intro-Steiner}. In Section \ref{sec:convexity} we establish some convexity properties which we will need for the proof of Theorem \ref{thm:intro-local}. In Section \ref{sec:equality1} we provide a proof of the equality cases of Theorems \ref{thm:main} and \ref{thm:intro-Steiner} as well as Theorem \ref{thm:intro-local} in the range $1 \leq k \leq n-2$; the case $k=n-1$ is treated in Section \ref{sec:equality2}. In Section \ref{sec:AF} we study the $L^p$-moment quermassintegrals $\Q_{k,p}(K)$ and establish Theorem \ref{thm:intro-half-conj}; an interesting interpretation of the case $p=0$ as a sharp averaged Loomis--Whitney isoperimetric inequality is described in Subsection \ref{subsec:LW}. In Section \ref{sec:conclude} we provide some concluding remarks -- in Subsection \ref{subsec:compact} we discuss possible extensions of Theorem \ref{thm:main} to more general compact sets, and in Subsection \ref{subsec:Petty} we present a new simple proof of Petty's projection inequality.

\bigskip
\noindent
\textbf{Acknowledgments.} We thank Richard Gardner, Erwin Lutwak, Rolf Schneider and Gaoyong Zhang for their comments on a preliminary version of this manuscript. We also thank the anonymous referees for their careful reading of the manuscript and very useful remarks, which have greatly helped in honing the paper.

\section{Notation and preliminaries} \label{sec:notation}

For a real number $a \in \R$, denote $a_+ := \max(a,0)$ and $a_- := (-a)_+$ so that $a = a_+ - a_-$. Denote $\R_+ := [0,\infty)$ and $\R_- := (-\infty,0]$. 
\smallskip

Given a Euclidean space $E$, we denote by $B_E$ its Euclidean unit ball and by $\S(E) = \partial B_E$ the corresponding unit-sphere; when $E = (\R^n,\scalar{\cdot,\cdot})$ we write $B_2^n$ and $\S^{n-1}$, respectively. We write $|x|$ for the Euclidean norm $\sqrt{\scalar{x,x}}$. We denote by $G_{E,k}$ the Grassmannian of all $k$-dimensional linear subspaces of $E$; when $E = \R^n$, we simply write $G_{n,k}$. It is equipped with its $\SO(E)$-invariant Haar probability measure, which we denote by $\sigma_{E,k}$, or $\sigma_{n,k}$ when $E = \R^n$. Here $\SO(E)$ denotes the group of rotations on $E$, equipped with its invariant Haar probability measure $\sigma_{\SO(E)}$; when $E = \R^n$, we simply write $\SO(n)$. $G_{E,k}$ is equipped with its standard topology and manifold structure as an $\SO(E)$-homogeneous space. 
For a linear map $T$ on $E$ we write $T^*$ for its adjoint and $T^{-*}$ for the adjoint of its inverse if $T$ is invertible. 

\smallskip

We use $\L_E$ to denote the Lebesgue measure on a $k$-dimensional affine subspace $E$; when the latter is clear from the context, we will simply write $\L^k$. The $k$-dimensional Hausdorff measure is denotes by $\H^k$.
Recall that $P_E$ denotes orthogonal projection onto $E$. Given a compact set $K$ in $\R^n$, we use $|K|$ and $|P_E K|$ as shorthand for $\L^n(K)$ and $\L^k(P_E K)$, respectively. 
\smallskip
The Steiner symmetral of a compact set $K \subset \R^n$ in the direction of $u \in \S^{n-1}$, denoted $S_u K$, is defined by requiring that the one-dimensional fiber $S_u K \cap (y + u^{\perp})$ is a symmetric interval about $u^{\perp}$ having the same one-dimensional Lebesgue measure as $K \cap (y + u^{\perp})$, for each $y \in u^{\perp}$ so that the latter is non-empty. Clearly $|S_u K| = |K|$, and it is well known that Steiner symmetrization preserves compactness as well as convexity \cite[Chapter 2]{GardnerGeometricTomography2ndEd}, \cite[Chapter 9]{Gruber-ConvexAndDiscreteGeometry}. We denote by $R_u$ the reflection map about $u^{\perp}$. 

\smallskip
The support function $h_K$ and polar body $K^{\circ}$ of a non-empty compact set $K \subset \R^n$ are defined as:
\[
h_K(x) := \max_{y \in K} \scalar{x,y} ~,~ K^{\circ} := \{ x \in \R^n \; ; \; h_K(x) \leq 1 \} .
\]
When in addition $K$ is convex and contains the origin in its interior, we define:
\[
\norm{x}_K := \inf \{ t > 0 \; ; \; x \in t K \} , 
\]
so that $K$ is precisely the unit ball of $\norm{\cdot}_K$. Note that in that case $\norm{x}_{K} = h_{K^\circ}(x)$ and that $(K^{\circ})^{\circ} = K$ \cite[Theorem 1.6.1]{Schneider-Book-2ndEd}. 

\smallskip
The family of convex compact non-empty sets in $E$ is denoted by $\K(E)$. An element of $\K(E)$ with non-empty interior in $E$ is called a ``convex body" in $E$. We use $\interior K$ to denote the (relative) interior in $E$. 

\smallskip
While we will not require this for the sequel, we recall for completeness several notions mentioned in the Introduction. The projection body $\Pi K$ of a convex body $K$, introduced (and shown to exist) by Minkowski, is defined as the convex body whose support function satisfies:
\[
h_{\Pi K}(\theta) = |P_{\theta^{\perp}} K|\;\;\; \forall \theta \in \S^{n-1}  . 
\]
The polar projection body $\Pi^* K$ is defined as $(\Pi K)^{\circ}$. The Santal\'o point $s(K)$ of $K$ is defined as the unique point $s$ in the interior of $K$ for which $|(K-s)^{\circ}|$ is minimized; $K^{\circ,s}$ is then defined as $s(K) + (K - s(K))^{\circ}$. We refer to \cite{Lutwak-Selected} and the references therein for further details and context. 

\smallskip
Recall that the Minkowski sum of two compact sets $A,B \subset \R^n$ is defined as $A + B := \{ a + b \; ; \; a \in A , b \in B\}$. It is immediate to see that $h_{A+B} = h_{A} + h_{B}$.
The Hausdorff distance between two compact subsets $A,B$ of $\R^n$ is defined as the minimal $\eps > 0$ so that $A \subseteq B + \eps B_2^n$ and $B \subseteq A + \eps B_2^n$. We say that $K$ has a point of symmetry (at $v \in \R^n$) if $K - v = -(K-v)$.

\smallskip 
The classical Brunn--Minkowski inequality \cite[Section 7.1]{Schneider-Book-2ndEd}, \cite{GardnerSurveyInBAMS}, \cite[Section B.2]{GardnerGeometricTomography2ndEd}, \cite[Chapter 8]{Gruber-ConvexAndDiscreteGeometry} states that if $K,L \in \K(\R^n)$ then
\[
|K+L|^{\frac{1}{n}} \geq |K|^{\frac{1}{n}} + |L|^{\frac{1}{n}} ,
\]
with equality when either $K$ or $L$ have non-empty interior if and only if $L = \lambda K + v$ for some $\lambda > 0$ and $v \in \R^n$. An equivalent dimension-free form of the inequality states that for any $\lambda \in [0,1]$:
\[
|(1-\lambda) K + \lambda L| \geq |K|^{1-\lambda} |L|^\lambda .
\]
Yet another equivalent form is given by Brunn's concavity principle \cite[Theorem 8.4]{Gruber-ConvexAndDiscreteGeometry}, stating that if $K$ is a convex body in $\R^{n+1}$ and $u \in \S^n$ then:
\[
\R \ni t \mapsto |K \cap (t u + u^{\perp})|^{\frac{1}{n}} \text{ is concave on its support.}
\]

\medskip

It was shown by Minkowski that when $\set{K_i}_{i=1}^m \subset \K(\Real^n)$, then the volume of their Minkowski sum is a polynomial with non-negative coefficients in the scaling parameters:
\[
|\sum_{i=1}^m t_i K_i| = \sum_{1 \leq i_1,\ldots,i_n \leq m} t_{i_1} \cdot \ldots \cdot t_{i_n} V(K_{i_1},\ldots,K_{i_n}) \;\;\; \forall t_i \geq 0 .
\]
The coefficient $V(K_{i_1},\ldots,K_{i_n}) \geq 0$ is called the mixed volume of the $n$-tuple $(K_{i_1},\ldots,K_{i_n})$; it is clearly multi-linear in its arguments (with respect to Minkowski addition and scaling by non-negative coefficients), and uniquely defined by requiring that it be invariant under permutations \cite[Section 5.1]{Schneider-Book-2ndEd}. Moreover, the mixed volume is continuous with respect to (joint) convergence of its arguments in the Hausdorff metric \cite[p. 280]{Schneider-Book-2ndEd}. 

Clearly $V(K,\ldots,K) = |K|$. Given $K,L \in \K(\R^n)$, we will use the abbreviation:
\[
V(K,k ; L, n-k) = V(\underbrace{K,\ldots,K}_{\text{$k$ times}},  \underbrace{L,\ldots,L}_{\text{$n-k$ times}}) . 
\]

The following formula is due to Fedotov (see \cite[Theorem 5.3.1]{Schneider-Book-2ndEd}) -- if $L_1,\ldots,L_{n-k} \in \K(F^{\perp})$ for some $F \in G_{n,k}$ ($k=1,\ldots,n-1$), then for any $K_1,\ldots,K_k \in \K(\R^n)$ we have:
\begin{equation} \label{eq:Fedotov}
{n \choose k} V(K_1,\ldots,K_k, L_1, \ldots, L_{n-k}) = V_{F}(P_F K_1,\ldots,P_F K_k) V_{F^{\perp}}(L_1,\ldots,L_{n-k}) ,
\end{equation}
where $V_E$, $E \in \{ F , F^{\perp} \}$, denotes the mixed volume in the subspace $E$.

\subsection{Shadow Systems} 
Steiner symmetrization can be viewed as a particular case of the more general construction of \emph{shadow systems}.  
We refer to Shephard \cite{Shephard-ShadowSystems} and Schneider \cite[Section 10.4]{Schneider-Book-2ndEd} for a general treatise, and only expand on what we need for this work. 

\smallskip

Given $u \in \S^{n-1} \subset \R^n$, let $T^u_t : \R^{n+1} \rightarrow \R^{n}$ denote the (non-orthogonal when $t \neq 0$) projection onto $\R^n$ parallel to $e_{n+1} + t u$. A family of convex compact sets $\{ K(t) \}_{t \in \R} \subset \K(\R^n)$ is called a \emph{shadow system in the direction of $u$} if there exists $\tilde K \in \K(\R^{n+1})$ so that $K(t) = T^u_t(\tilde K)$ for all $t$. This definition was introduced (in a more general form) by Shephard \cite{Shephard-ShadowSystems}, who noted the equivalence of this particular instance with the prior definition by Rogers--Shephard  \cite{RogersShephard-ShadowSystems} of a \emph{linear parameter system}; we will not use the latter terminology here. We will sometimes omit the index set $\R$ and simply write $\{K(t)\}$. 

\smallskip

\begin{figure}
\centering
\begin{tikzpicture}

 \draw[<->, thin, black] (-7,0) -- (7,0);
  \draw[->, thin, black] (0,0) -- (0,7);
 \draw  (7.4,0)  node {$\R^n$};
 \draw  (0,7.4)  node {$\R$};

 \draw [line width=5pt,gray!80] (-6,0) -- (-2.6,0);
 \draw  (-4.6,-0.4)  node {$K_u(1) = K$};
  \draw[ thin, dotted] (-7,-1) -- (1,7);
  \draw[thin, dotted] (-3.6,-1) -- (4.4,7);

 \draw [line width=5pt,gray!80] (6,0) -- (2.6,0);
 \draw  (4.7,-0.4)  node {$K_u(-1) = R_uK$};
  \draw[ thin, dotted] (7,-1) -- (-1,7);
  \draw[ thin, dotted] (3.6,-1) -- (-4.4,7);

 \draw [line width=5pt,gray!80] (-1.7,0) -- (1.7,0);
 \draw  (0,-0.4)  node {$K_u(0) = S_uK$};
  \draw[ thin, dotted] (-1.7,-1) -- (-1.7,7);
  \draw[ thin, dotted] (1.7,-1) -- (1.7,7);

\filldraw[fill=gray!80, draw = gray!80] (0,6) -- (-1.7,4.3) -- (0,2.6) -- (1.7,4.3) -- (0,6);
\draw  (1,5.5)  node {$\tilde K$};

  \draw [->,line width=1.5pt] (0,0) -- (0,2.2);
 \draw  (0.6,1.2)  node {$e_{n+1}$};
 \draw [->,line width=1.5pt] (0,0) -- (2.2,0);
 \draw  (1.2,0.4)  node {$u$};

\end{tikzpicture}

\caption{
          \label{fig:shadow}
         A two-dimensional slice of the shadow system.
     }
\end{figure}

It was observed by Shephard in \cite[(4)]{Shephard-ShadowSystems} that given $K_0,K_1 \in \K(\R^n)$, there exists a shadow system $\{K(t)\}$ in the direction of $u$ so that $K(t_0) = K_0$ and $K(t_1) = K_1$ for some (equivalently, any) 
$t_0 \neq t_1 \in \R$, if and only if $P_{u^{\perp}} K_0 = P_{u^{\perp}} K_1$. Moreover, fixing $t_0 \neq t_1 \in \R$, there exists a maximal shadow system $\{K_{\max}(t)\}$ with this property, in the sense that $K_{\max}(t) \supseteq K(t)$ for any $\{K(t)\}$ as above; it was shown by Shephard that it is given by $K_{\max}(t) = T^u_t(\tilde K_{\max})$ with:
\begin{equation} \label{eq:maximal-rep}
 \tilde K_{\max} := (T^u_{t_0})^{-1}(K_{0}) \cap (T^u_{t_1})^{-1}(K_{1}) .
 \end{equation}
 Equivalently, if we denote by $A^{(y)}$ the one-dimensional section of $A$ in the direction of $u$ over $y \in u^{\perp}$, we have:
 \begin{equation} \label{eq:linear-shadow-def}
 (K_{\max}((1-\alpha) t_0 + \alpha t_1))^{(y)} := (1-\alpha) (K_{0})^{(y)} + \alpha (K_{1})^{(y)} \;\;\; \forall \alpha \in \R \;\;\; \forall y \in u^{\perp} .
 \end{equation}
This maximal shadow system is called a \emph{linear} shadow system by Shephard \cite[(4)]{Shephard-ShadowSystems}, not to be confused with Minkowski's linear system of convex bodies \cite[p. 36]{BonnesenFenchelBook}, nor with the Rogers--Shephard definition of a linear parameter system \cite{RogersShephard-ShadowSystems} (which, as we have already remarked, is equivalent to that of a general, possibly non-maximal, shadow system). 

\smallskip

In this work, we will mostly be concerned with a specific \emph{linear} (i.e.~maximal) shadow system constructed from $K$ and $R_u K$, where recall $R_u$ denotes reflection about $u^{\perp}$. Since $P_{u^{\perp}} R_u K = P_{u^{\perp}} K$, it follows that there exists a linear shadow system $\{K_u(t)\}$ in the direction of $u$ so that $K_u(1) = K$ and $K_u(-1) = R_u K$; see Figure \ref{fig:shadow}. By (\ref{eq:linear-shadow-def}), it is given by:
\begin{equation} \label{eq:linear-reflection-shadow-def}
 (K_u(t))^{(y)} := \frac{1+t}{2} K^{(y)} + \frac{1-t}{2} (R_u K)^{(y)} \;\;\; \forall t \in \R \;\;\; \forall y \in u^{\perp} .
\end{equation}
The following is immediate:

\begin{lemma}[Definition and Properties of $\{K_u(t)\}$] \label{lem:reflection-shadow}
Given a convex body $K \subset \R^n$ and a direction $u \in \S^{n-1}$, $\{K_u(t)\}_{t \in \R}$ defined by (\ref{eq:linear-reflection-shadow-def}) is a linear shadow system of convex bodies in $\R^n$ with the following properties:
\begin{itemize}
\item $K_u(1) = K$ and $K_u(-1) = R_u K$. 
\item More generally, $R_u (K_u(t)) = K_u(-t)$ for all $t \in \R$.
\item $K_u(0) = S_u K$. 
\item $|K_u(t)| = |K|$ for all $t \in [-1,1]$ (but not outside this interval!). 
\item $\R \ni t \mapsto K_u(t)$ is continuous in the Hausdorff topology. 
\end{itemize}
$\{K_u(t)\}$ is called the linear reflection shadow system associated to $K$ in the direction of $u$. 
\end{lemma}

We conclude that $\{K_u(t)\}$ continuously interpolates between $K$, $S_u K$ and $R_u K$ at times $t=1$, $0$ and $-1$, respectively. 
For a general shadow system $\{K(t)\}$ in the direction of $u$, the one-dimensional set-equality (\ref{eq:linear-shadow-def}) is replaced by set-inclusion: 

\begin{lemma}
For any shadow system $\{K(t)\}_{t \in \R}$ in the direction of $u$ and all $y \in u^{\perp}$:
\begin{equation} \label{eq:shadow-system-convex}
(K((1-\alpha) t_0 + \alpha t_1))^{(y)} \subseteq (1-\alpha) (K(t_0))^{(y)} + \alpha  (K(t_1))^{(y)} \;\;\; \forall \alpha \in \R \;\; \forall t_0,t_1 \in \R .
\end{equation}
\end{lemma}
\begin{proof}
For $t_0 = t_1$ there is nothing to prove, so fix $t_0 \neq t_1 \in \R$ and denote $K_{i} = K(t_i)$, $i=0,1$. By Shephard's characterization of shadow systems, $P_{u^{\perp}} K_0 = P_{u^{\perp}} K_1$, and we may consider the corresponding maximal shadow system $K_{\max}(t) =  T^u_t(\tilde K_{\max})$ with $\tilde K_{\max}$ given by (\ref{eq:maximal-rep}), for which (\ref{eq:linear-shadow-def}) holds. 
By maximality, $K(t) \subseteq K_{\max}(t)$ for all $t \in \R$, and the assertion follows. 
See also \cite[(7)]{RogersShephard-ShadowSystems} or \cite[(4)]{MeyerReisner-SantaloViaShadowSystems}.
\end{proof}

\section{Proof of the isoperimetric inequality} \label{sec:main}

The inequality of Conjecture \ref{conj:Lutwak} is a standard consequence of the following symmetrization result, already stated as (\ref{eq:intro-Steiner}) in Theorem \ref{thm:intro-Steiner}:

\begin{thm} \label{thm:Steiner}
Steiner symmetrization in a direction $u \in \S^{n-1}$ does not increase the $k$-th affine quermassintegral $\Phi_k(K)$ for any convex body $K \subset \R^n$ and $k=1,\ldots,n-1$:
\[
\Phi_k(K) \geq \Phi_k(S_u K) .
\]
\end{thm}
 
Throughout the ensuing sections,  $u \in \S^{n-1}$ denotes a fixed vector (the direction in which the Steiner symmetrization is performed), $k=1,\ldots,n-1$ is fixed, $E \in G_{u^{\perp},k-1}$ denotes a $(k-1)$-dimensional linear subspace in $u^{\perp}$, and $x$ denotes a vector in $E^{\perp}$. 

\subsection{Ingredients}

For the proof, we will need three main ingredients. Most of these ingredients also make sense when $k=n$, but yield trivial constructions or statements, and so we omit this extraneous generality here. 

\subsubsection{The Projection Rolodex}

In the initial parts of this subsection, we assume that $E \in G_{n,k-1}$ and $x \in \R^n$, and only later restrict to $E \in G_{u^{\perp},k-1}$ and $x \in E^{\perp}$. Given $K \in \K(\R^n)$, denote:
\begin{equation} \label{eq:PEwedge}
|P_{E \wedge x} K| := |P_{E^{\perp}} x| \L^k(P_{\sspan(E , x)} K) .
\end{equation}
We will mainly consider the case when $x \in E^{\perp}$, so that $|P_{E^{\perp}} x| = |x|$. By $|P_x K|$ we will mean $|x| |P_{\sspan(x)} K|$, corresponding to the case $E = \{0\}$ above. 

\begin{lemma} \label{lem:proj-vol-continuous}
The mapping $G_{n,k-1} \times \R^n \ni (E,x) \mapsto |P_{E \wedge x} K| \in \R_+$ is continuous.
\end{lemma}
\begin{proof}
By (\ref{eq:Fedotov}) applied to $F = \sspan(E , x)$, whenever $x \notin E$:
\[
{n \choose k} V(K, k ; |P_{E^{\perp}} x|^{\frac{1}{n-k}} B_{E^{\perp} \cap x^\perp} , n-k) = \L^k(P_{\sspan(E,x)} K) |P_{E^{\perp}} x| |B_{F^{\perp}}| 
\]
(recall that $B_{F^{\perp}}$ denotes the unit Euclidean ball in the subspace $F^{\perp}$).
It follows that whenever $x \notin E$:
\[
|P_{E \wedge x} K| = \frac{{n \choose k}}{|B_2^{n-k}|} V(K, k ; |P_{E^{\perp}} x|^{\frac{1}{n-k}} B_{E^{\perp} \cap x^\perp} , n-k) ,
\]
and this remains valid also when $x \in E$ since in that case both sides are equal to zero. Clearly $(E,x) \mapsto |P_{E^{\perp}} x|^{\frac{1}{n-k}} B_{E^{\perp} \cap x^\perp}$ is continuous with respect to the Hausdorff topology in a neighborhood of $(E_0,x_0)$ when $x_0 \notin E_0$, but also when $x_0 \in E_0$ thanks to the prefactor $|P_{E^{\perp}} x|^{\frac{1}{n-k}}$ (which tends to zero as $(E,x) \rightarrow (E_0,x_0)$). Consequently, the continuity of mixed volume with respect to the Hausdorff topology concludes the proof. 
\end{proof}

We introduce the following two definitions, which may be of independent interest:
\begin{defn} \label{defn:E-polar-body}
Given $K \in \K(\R^n)$ and $E \in G_{n,k-1}$, the set
\[
L_E(K) := \{ x \in E^{\perp} \; ; \; |P_{E \wedge x} K| \leq 1 \} \subset E^{\perp} 
\]
is called the $E$-projected polar body of $K$. 
\end{defn}

Note that $L_E(K)$ is always origin-symmetric, closed and contains the origin in its interior (as $K$ is compact). Whenever $K$ has non-empty interior then $L_E(K)$ is in addition bounded in $E^{\perp}$ and hence compact. A less obvious property is that $L_E(K)$ is always convex, as we shall momentarily show. We conclude that whenever $K$ is a convex body in $\R^n$, the $E$-projected polar body $L_E(K)$ is an origin-symmetric convex body in $E^{\perp}$.

\begin{lemma} \label{lem:convex}
Let $K \in \K(\R^n)$. For any $E \in G_{n,k-1}$, the map $\R^n \ni x \mapsto |P_{E \wedge x} K|$ is convex. In particular, its level set  $L_E(K)$ in $E^{\perp}$ is convex. 
\end{lemma}
\noindent
This is immediate to see when $E = \{0\}$, in which case $|P_x K| = h_K(x) + h_K(-x)$ and
\[
L(K) := L_{\{0\}}(K) = \{ x \in \R^n \; ; \;  h_K(x) + h_K(-x) \leq 1 \}  = (K-K)^{\circ} .
\]
Hence, when $K$ is an origin-symmetric convex body, $L(K)$ coincides with $\frac{1}{2} K^{\circ}$. Note that contrary to the usual definition of polar body of a convex set, the above definition is invariant under translations of $K$.  

To treat more general subspaces $E$, we introduce some additional useful notation, which will also be used in subsequent sections. Given $w \in P_E K$, denote:
\[
 K^w := (K - w) \cap E^{\perp},
 \]
 and note that if $x \in E^{\perp}$ then:
\[
P_{\sspan(E, x)} K = \bigsqcup_{w \in P_E K} (w + P_{\sspan(x)} K^w) . 
\]
Hence by Fubini and homogeneity, for all $x \in E^{\perp}$: 
\begin{equation} \label{eq:Fubini}
 |P_{E \wedge x} K| = \int_{P_E K} |P_{x} K^w| dw = \int_{P_E K} (h_{K^w}(x) + h_{K^w}(-x)) dw  .
\end{equation}
Note that (\ref{eq:Fubini}) yields a useful expression for $\norm{x}_{L_E(K)}$. 

\begin{proof}[Proof of Lemma \ref{lem:convex}]
Since $|P_{E \wedge x} K|$ only depends on $P_{E^{\perp}} x$, it is enough to establish convexity for $x \in E^{\perp}$. But this is now immediate from (\ref{eq:Fubini}) and  the convexity of the support functions $h_{K^w}$. 
\end{proof}

\medskip

We will henceforth assume that $E \in G_{u^{\perp},k-1}$. 
It will be useful to introduce the following notation given $s \in \R$: 
\begin{equation} \label{eq:LEus-def}
L_{E,u,s}(K)  := \{y  \in E^{\perp} \cap u^{\perp} \; ; \; |P_{E\wedge (y + su)} K| \leq 1\} .
\end{equation}
Note that $L_{E,u,s}(K)$ is the section of $L_E(K)$ perpendicular to $u$ at height $s$, and therefore itself convex. Furthermore, by Brunn's concavity principle, the map $\R \ni s \mapsto |L_{E,u,s}(K)|^{\frac{1}{n-k}}$ is concave on its support, and so in particular $\R \ni s \mapsto |L_{E,u,s}(K)|$ is measurable. 

\medskip

We denote by $V_{k,u}$ the following vector bundle over $G_{u^{\perp} , k-1}$: 
\[
V_{k,u} := \{ (E,x) \; ; \; E \in G_{u^{\perp} , k-1} , x \in E^{\perp} \} ,
\]
equipped with the subspace topology as a closed subset of $G_{u^{\perp},k-1} \times \R^n$. 
Given $K \in \K(\R^n)$, the mapping:
\[
V_{k,u} \ni (E,x) \mapsto |P_{E \wedge x} K| \in \R_+ 
\]
is continuous as the restriction of the continuous mapping from Lemma \ref{lem:proj-vol-continuous}, and hence its sub-level set
\[
\{ (E, x) \in V_{k,u} \; ;\; |P_{E \wedge x} K| \leq 1\} = \{ (E , x) \; ; \; E \in G_{u^{\perp} , k-1} , x \in L_E(K) \}
\]
is a closed subset of $V_{k,u}$. When in addition $K$ has non-empty interior, this subset is bounded and hence compact. 

\begin{defn}
Given $K \in \K(\R^n)$, the closed subset 
\[
L_{k,u}(K) := \{ (E , x) \; ; \; E \in G_{u^{\perp} , k-1} , x \in L_E(K) \} \subset V_{k,u}
\]
is called the $k$-dimensional Projection Rolodex of $K$ relative to $u^{\perp}$. 
\end{defn}

The idea behind this definition is that it encodes the values of $|P_F K|$ for $\sigma_{n,k}$-almost-every $F \in G_{n,k}$; indeed, we may write almost-every $F \in G_{n,k}$ as the direct sum of $E = F \cap u^{\perp} \in G_{u^{\perp},k-1}$ and $\sspan(\theta)$ for $\theta \in \S(E^{\perp})$, and use that $t \theta \in L_E(K)$ iff $|t| \leq 1/ |P_F K|$, i.e.~that:
\begin{equation} \label{eq:rolodex-prop}
 |P_F K| = |P_{\sspan(E,\theta)} K| = \norm{\theta}_{L_E(K)} \;\;\; \forall \theta \in \S(E^{\perp}). 
\end{equation}

\subsubsection{Convexity of shadow system's projections} 

Our second main ingredient is the following key proposition, which pertains to a certain convexity property of projections of shadow systems.  

\begin{proposition} \label{prop:key}
Let $\{K(t) \}_{t \in \R}$ denote a shadow system in the direction of $u \in \S^{n-1}$, and let 
 $E \in G_{u^{\perp} , k-1}$. Then for any fixed $s \in \R$, the function:
\[
u^{\perp} \times \R \ni (y,t) \mapsto|P_{E \wedge (y + su)} K(t)| 
\]
is jointly convex in $(y,t)$. 
\end{proposition}

While we will only require this for the linear reflection shadow system $\{K_u(t) \}_{t \in [-1,1]}$, the proof extends to general shadow systems with no additional effort. The proof of Proposition \ref{prop:key} is deferred to Subsection \ref{subsec:key1}; an alternative proof is presented in Section \ref{sec:convexity}.

\subsubsection{Blaschke--Petkantschin-type formula} 

The final crucial ingredient, without which we do not know how to obtain sharp lower bounds on $\Q_{k,p}(K)$ for $p < -(k+1)$ (see Remark \ref{rem:Petty-stuck}), is the following Blaschke--Petkantschin-type formula. 

\begin{thm} \label{thm:BP}
Fix $u \in \S^{n-1}$. There exists a constant $c_{n,k} > 0$ so that for any measurable function $f : G_{n,k} \rightarrow \R_+$:
\[
c_{n,k} \int_{G_{n,k}} f(F) \sigma_{n,k}(dF) =  \int_{G_{u^{\perp} , k-1}} \int_{\S^{n-k}(E^{\perp})} f(\sspan(E,\theta)) \abs{\scalar{\theta,u}}^{k-1}  d\theta \, \sigma_{u^{\perp},k-1}(dE) ,
\]
where $\sigma_{u^{\perp},k-1}$ is the uniform Haar probability measure on $G_{u^{\perp} , k-1}$. 
\end{thm}
\begin{proof}
The statement is a particular case of \cite[Theorem 7.2.6]{SchneiderWeil-Book} from the excellent monograph of Schneider--Weil, applied with $d=n$, $q=1$, $m = s_1 = k$, $s_0 = d-1$ and $L_0 = u^{\perp}$. To be a bit more precise, in \cite[Theorem 7.2.6]{SchneiderWeil-Book}, the inner integral on the right hand side above is given by:
\[
\int_{G(E,k)} f(F) [F,u^{\perp}]^{k-1} \sigma_{E,k}(dF) ,
\]
where $G(E,k)$ denotes the homogeneous space $\{ F \in G_{n,k} \; ; \; F \supset E \}$ equipped with its uniform Haar probability measure $\sigma_{E,k}$ (invariant under the action of $SO(n,E) := \{ U \in SO(n) \; ; \; U E = E \}$), and $[F,u^{\perp}]$ is the subspace determinant defined in \cite[Section 14.1]{SchneiderWeil-Book}. Any $F \in G(E,k)$ may be written as $F = \sspan(E,\theta)$ for some $\theta \in \S^{n-k}(E^{\perp})$, and since the Haar measure $d \theta$ on $\S^{n-k}(E^{\perp})$ is invariant under the action of $SO(n,E)$, uniqueness of the Haar measure (up to a multiplicative constant) implies that we may rewrite this integral as:
\begin{equation} \label{eq:extra-factor}
\frac{1}{|\S^{n-k}|} \int_{\S^{n-k}(E^{\perp})} f(\sspan(E,\theta)) [\sspan(E,\theta), u^{\perp}]^{k-1} d\theta . 
\end{equation}
By definition, $ [\sspan(E,\theta), u^{\perp}] = [\sspan(E,\theta)^{\perp} , \sspan(u)]$, where $[F_1,F_2]$ denotes, when $b = \dim F_1 + \dim F_2 \leq n$, the $b$-dimensional volume of the parallelepiped spanned by the union of any orthonormal bases of $F_1$ and $F_2$. It remains to note that since $E \subset u^{\perp}$ and $\theta \in E^{\perp}$:
\[
[\sspan(E,\theta)^{\perp} , \sspan(u)] = [\sspan(E,\theta)^{\perp} \oplus E , \sspan(u)] =  [\sspan(\theta)^{\perp},\sspan(u)] = \abs{\scalar{\theta,u}} .
 \]
 The value of $c_{n,k}$ is immaterial for us, but may be found in \cite[Theorem 7.2.6]{SchneiderWeil-Book}, after taking into account the extra $|\S^{n-k}| = \H^{n-k}(\S^{n-k})$ term we incurred in (\ref{eq:extra-factor}). 
 \end{proof}

\subsection{Proof of Theorem \ref{thm:Steiner}}

Given $u \in \S^{n-1}$, introduce the following Borel measure on $V_{k,u}$
\[
 \mu_{k,u}(dE,dx) := \abs{\scalar{x,u}}^{k-1} \L_{E^{\perp}}(dx) \sigma_{u^{\perp} , k-1}(dE) .
\]
To be slightly more precise, $\mu_{k,u}$ is obtained as the restriction of the Borel product measure $\sigma_{u^{\perp} , k-1}(dE) \otimes (\abs{\scalar{x,u}}^{k-1} \H^{n-k+1}(dx))$ on $G_{u^{\perp},k-1} \times \R^n$ to the closed subset $V_{k,u}$, and thus defines a measure on the Borel $\sigma$-algebra on $V_{k,u}$. 
Recall that $L_{k,u}(K)$ is a closed subset of $V_{k,u}$ and therefore Borel measurable.

\begin{lemma} \label{lem:mu}
For any  $K \in \K(\R^n)$ and $u \in \S^{n-1}$:
\begin{equation} \label{eq:mu}
\mu_{k,u}(L_{k,u}(K)) = \frac{c_{n,k}}{n} \int_{G_{n,k}} \frac{1}{|P_F K|^n} \sigma_{n,k}(dF) . 
\end{equation}
\end{lemma}
\begin{proof}
Set $p(x) := \abs{\scalar{x,u}}^{k-1}$. Integrating in polar coordinates on $E^{\perp}$ and  invoking Theorem \ref{thm:BP}, we obtain:
\begin{align*}
\mu_{k,u}(L_{k,u}(K)) & = \int_{G_{u^{\perp} , k-1}}  \int_{E^{\perp}} 1_{L_{k,u}(K)}(E,x) p(x) \L_{E^{\perp}}(d x) \sigma_{u^{\perp} , k-1}(dE) \\
& = \int_{G_{u^{\perp} , k-1}}  \int_{\S^{n-k}(E^{\perp})} \int_0^\infty 1_{L_{k,u}(K)}(E,r \theta) p(r \theta) r^{n-k} dr \, d\theta \, \sigma_{u^{\perp} , k-1}(dE) \\
& = \int_{G_{u^{\perp} , k-1}}  \int_{\S^{n-k}(E^{\perp})} p(\theta) \int_0^{1 / |P_{\sspan(E,\theta)}(K)|}  r^{n-1} dr \, d\theta \,\sigma_{u^{\perp} , k-1}(dE) \\
& = \frac{1}{n} \int_{G_{u^{\perp} , k-1}}  \int_{\S^{n-k}(E^{\perp})} \frac{1}{|P_{\sspan(E,\theta)} K|^n} \abs{\scalar{\theta,u}}^{k-1} d\theta \, \sigma_{u^{\perp} , k-1}(dE) \\
& = \frac{c_{n,k}}{n} \int_{G_{n,k}} \frac{1}{|P_F K|^n} \sigma_{n,k}(dF) .
\end{align*}
\end{proof}
\begin{rem} \label{rem:finite}
While the identity (\ref{eq:mu}) is valid for an arbitrary $K \in \K(\R^n)$ by the Fubini-Tonelli theorem, note that when $K$ has empty interior, the expressions in (\ref{eq:mu}) may (both) be infinite. However,  both expressions will be finite whenever $K$ has non-empty interior, i.e. when $K$ is a convex body.  
\end{rem}

\begin{proof}[Proof of Theorem \ref{thm:Steiner}]
In view of Lemma \ref{lem:mu}, we would like to show that for any convex body $K$:
\begin{equation} \label{eq:mu-goal}
\mu_{k,u}(L_{k,u}(K)) \leq \mu_{k,u}(L_{k,u}(S_u K)) . 
\end{equation}
The advantage of the latter formulation is that now everything is ``aligned" with $u$, the direction in which we perform the Steiner symmetrization. Consequently, we evaluate things by decomposing each $E^{\perp}$ into $\sspan(u) \oplus (E^{\perp} \cap u^{\perp})$ and applying Fubini:
\begin{align}
\nonumber \mu_{k,u}(L_{k,u}(K)) & = \int_{G_{u^{\perp} , k-1}}  \int_{E^{\perp}} 1_{L_{k,u}(K)}(E,x) \abs{\scalar{x,u}}^{k-1} \L_{E^{\perp}}(d x) \sigma_{u^{\perp} , k-1}(dE) \\
\nonumber & = \int_{G_{u^{\perp} , k-1}}  \int_\R \int_{E^\perp \cap u^\perp} 1_{|P_{E \wedge (y + su)} K| \leq 1} \abs{\scalar{y + s u,u}}^{k-1}  dy \, ds \, \sigma_{u^{\perp},k-1}(dE) \\
\nonumber  & = \int_{G_{u^{\perp} , k-1}} \int_\R |s|^{k-1} \int_{E^\perp \cap u^\perp} 1_{|P_{E \wedge (y + su)} K| \leq 1}  dy \, ds \, \sigma_{u^{\perp},k-1}(dE) \\
\label{eq:punch} & = \int_{G_{u^{\perp} , k-1}} \int_\R |s|^{k-1} |L_{E,u,s}(K)| \, ds \, \sigma_{u^{\perp},k-1}(dE) ,
\end{align}
where $L_{E,u,s}(K)$ was defined in (\ref{eq:LEus-def}), and shown there to be convex (and hence measurable), and $\R \ni s \mapsto |L_{E,u,s}(K)|$ was shown to be measurable as well. Note that the Borel measurability of the inner integral in $E \in G_{u^{\perp} , k-1}$ follows directly from the Fubini-Tonelli theorem, applied to the iterated integral of the Borel function $1_{L_{k,u}(K)}$ with respect to the Borel product measure $\sigma_{u^{\perp} , k-1}(dE) \otimes (\abs{\scalar{x,u}}^{k-1} \H^{n-k+1}(dx))$.

\smallskip

So far we haven't used the convexity of $K$ in any essential way. We now apply the key Proposition \ref{prop:key} to the linear reflection shadow system $\{K_u(t)\}$ from Lemma \ref{lem:reflection-shadow}. 
As $E \subset u^{\perp}$ and $y\in E^{\perp} \cap u^{\perp}$, it follows that for every fixed $s \in \R$, the function:
\[
 (E^{\perp} \cap u^{\perp}) \times \R \ni (y,t) \mapsto  f^{(s)}(y,t) := |P_{E \wedge (y + su)} K_u(t)| \;\;\; \text{is jointly convex.}
\]
In addition, $f^{(s)}(y,t)$ is an even function, since $K_u(-t) = R_u(K_u(t))$ and hence:
\begin{align}
\label{eq:flip} f^{(s)}(-y,-t)&  = |P_{E \wedge (-y + su)} K_u(-t)| = |P_{R_u E \wedge R_u(-y + su)}  K_u(t)| \\
\nonumber & = |P_{E \wedge (-y - su)} K_u(t)| = |P_{E \wedge (y + su)} K_u(t)| = f^{(s)}(y,t) . 
\end{align}
Consequently, its level set:
\[
 \tilde L_{E,u,s} := \{ (y,t) \in  (E^{\perp} \cap u^{\perp}) \times \R   \; ; \; |P_{E \wedge (y + su)} K_u(t)| \leq 1 \}
\]
is an origin-symmetric convex body. Note that its $t$-section is precisely $L_{E,u,s}(K_u(t))$. Inspecting the $t$-sections at $t=-1,0,1$ and recalling that $K_u(1) =  K$ and $K_u(0) = S_u K$, convexity and origin-symmetry of $\tilde L_{E,u,s}$ imply:
\begin{equation} \label{eq:key-inq}
L_{E,u,s}(S_u K) \supseteq \frac{1}{2}(L_{E,u,s}(K) - L_{E,u,s}(K)) . 
\end{equation}
By the Brunn--Minkowski inequality, we deduce:
\[
|L_{E,u,s}(S_u K)| \geq |L_{E,u,s}(K)| 
\]
(which remains valid also when $L_{E,u,s}(K) = \emptyset$). 
Plugging this back into (\ref{eq:punch}) and rolling everything back, we deduce the desired (\ref{eq:mu-goal}), thereby concluding the proof. 
\end{proof}

In fact, the above proof gives us more information: \begin{thm} \label{thm:monotone} 
For any convex body $K$ in $\R^n$ and $u \in \S^{n-1}$, the function $\R_+ \ni t \mapsto \Phi_k(K_u(t)) = \Phi_k(K_u(-t))$ is monotone non-decreasing. 
\end{thm}
\begin{proof}
As $\dim (E^{\perp} \cap u^{\perp}) = n-k$, we actually know by Brunn's concavity principle, applied to the $t$-sections of $\tilde L_{E,u,s}$, that the function $\R \ni t \mapsto |L_{E,u,s}(K_u(t))|^{\frac{1}{n-k}}$ is concave on its support. It is also even by origin-symmetry of  $\tilde L_{E,u,s}$. In particular,
\begin{equation} \label{eq:VolLmonotone}
\R_+ \ni t \mapsto |L_{E,u,s}(K_u(t))| = |L_{E,u,s}(K_u(-t))| \text{ is non-increasing} .
\end{equation}
Integrating this according to (\ref{eq:punch}) and applying Lemma \ref{lem:mu}, the assertion follows. 

\end{proof}

\subsection{Proof of convexity of shadow system projections} \label{subsec:key1}

To complete the proof, it remains to establish Proposition \ref{prop:key}. 

\smallskip

Denote:
\[
|P_{x_1\wedge\ldots\wedge x_k} K| := \L^k(P_{\sspan\{x_1,\ldots,x_k\}} K) \Delta(x_1,\ldots,x_k) , 
\]
where $\Delta(x_1,\ldots,x_k)$ denotes the $\L^k$ measure of the parallelepiped $[0,x_1] + \ldots + [0,x_k]$. 
This is consistent with our previous notation introduced in (\ref{eq:PEwedge}) since if $E \in G_{n,k-1}$ is spanned by an orthonormal basis $\{x_1,\ldots,x_{k-1}\}$, we clearly have:
\[
|P_{E \wedge x_k} K| = |P_{x_1 \wedge \ldots \wedge x_{k-1} \wedge x_k} K| .
\]

\medskip

We will require the following linear-algebra lemma:
\begin{lemma} \label{lem:wedge}
For any linear map $T : \R^n \rightarrow \R^n$, $x_1,\ldots,x_k \in \R^n$ and compact set $A\subset \R^n$:
\begin{equation} \label{eq:wedge}
|P_{x_1\wedge\ldots\wedge x_k} T(A)| = |P_{T^*(x_1)\wedge \ldots \wedge T^*(x_k)} A| .
\end{equation}
\end{lemma}
\noindent
We will actually only require to know that $|P_{E \wedge x} T(K)| = |P_{E \wedge T^*(x)} K|$ when $T$ acts as the identity on $E$ and invariantly on $E^{\perp}$, which is totally elementary and may be proved as in Lemma \ref{lem:convex} by using that $h_{T(K^w)}(x) = h_{K^w}(T^* x)$; however, for completeness, we provide a proof of the general version above. 
First observe:
\begin{lemma} \label{lem:prepare-wedge}
For any linear map $T: \R^n \rightarrow \R^n$ and subspace $E \subseteq \R^n$ so that $T^*|_E : E \rightarrow T^* E$ is injective, there exists a linear map $S: T^* E \rightarrow E$ so that:
\begin{equation} \label{eq:S}
P_E \circ T = S  \circ P_{T^* E} . 
\end{equation}
\end{lemma}
\begin{proof}
The operator $M = T^* \circ P_E \circ T$ is self-adjoint so $\Img M \subseteq T^* E$ is an invariant subspace of $M$. Since $\Ker M = (\Img M)^{\perp} \supseteq (T^* E)^{\perp}$,  
we may therefore write:
\[
T^* \circ P_E \circ T = M =  N \circ P_{T^* E} ,
\]
for some self-adjoint linear map $N : T^* E \rightarrow T^* E$. It follows that (\ref{eq:S}) holds with $S = (T^*|_E)^{-1} \circ N$. 
\end{proof}

\begin{proof}[Proof of Lemma \ref{lem:wedge}]
We may assume that $\{x_i\}$ are linearly independent, otherwise both sides of (\ref{eq:wedge}) are zero and there is nothing to prove. Denote $E = \sspan\{x_1,\ldots,x_k\}$. 
Note that $P_E \circ T$ is onto $E$ iff $(\Img T)^{\perp} \cap E = \{ 0 \}$ iff $\Ker T^* \cap E = \{ 0 \}$, and so we may assume that $T^*|_E$ is injective, otherwise both sides of (\ref{eq:wedge}) are again zero. By Lemma \ref{lem:prepare-wedge}:
\[
|P_E \circ T(A)| = |S  \circ P_{T^* E}(A)| = |\det_{T^*E \rightarrow E} S | \cdot |P_{T^* E}(A)| , 
\]
where $|\det_{P \rightarrow Q} L|$, the (constant) Jacobian of the linear map $L : P \rightarrow Q$ between two isomorphic vector spaces $P$ and $Q$, is equal to $\sqrt{\det_{P \rightarrow P}  L^* L} = \sqrt{ \det_{Q \rightarrow Q} L L^*}$. Since $S S^* = S P_{T^* E} S^*$, we have by (\ref{eq:S}):
\[
|\det_{T^*E \rightarrow E} S| = \sqrt{\det_{E \rightarrow E} S S^*} = \sqrt{ \det_{E \rightarrow E} P_E T T^* P_E} = |\det_{E \rightarrow T^* E} T^*|_E |. 
\]
Hence:
\begin{align*}
& |P_{x_1\wedge\ldots\wedge x_k} T(A)| = \Delta(x_1,\ldots,x_k) |P_E T(A)| = \Delta(x_1,\ldots,x_k) |\det_{E \rightarrow T^* E} T^*|_E | \cdot |P_{T^* E}(A)|\\
& = \Delta(T^* x_1,\ldots,T^* x_k) |P_{T^* E}(A)|= |P_{T^*(x_1)\wedge \ldots \wedge T^*(x_k)} A|.
\end{align*}
\end{proof}

We are now ready to prove Proposition \ref{prop:key}
\begin{proof}[Proof of Proposition \ref{prop:key}]
We know that $K(t) = T^u_t(\tilde K)$ for some  $\tilde K \in \K(\R^{n+1})$, where $T^u_t$ is a projection which acts as identity on $\R^n$ and sends $e_{n+1}$ to $- t u$. 
One immediately checks that $(T^u_t)^*$ acts as the identity on $u^{\perp}$ (and in particular on $E \subset u^{\perp}$) and $(T^u_t)^*(u) = u - t e_{n+1}$.
Hence by Lemma \ref{lem:wedge}, if $y \in u^{\perp}$:
\[
|P_{E \wedge (s u + y)} K(t)| = |P_{E \wedge (s u - s t e_{n+1} + y)} \tilde K| .
\]
(see also \cite[(5)]{CampiGronchi-VolumeProductInqs} for the case $E = \{0\}$). 
The convexity in $(y,t)$ now follows from Lemma \ref{lem:convex}, as the map $(t,y) \mapsto su -s t e_{n+1} + y$ is affine (for fixed $s \in \R$). 
\end{proof}

\begin{rem} \label{rem:more-general}
We remark that the method described in this subsection is quite general, and may be used to show the following much more general version of Proposition \ref{prop:key}.
Let $\{K(t)\}_{t \in \R}$ be a shadow system in the direction of $u$. Then for any $x_1,\ldots,x_k \in u^{\perp}$, for any $a_1,\ldots,a_k \in \R$, and for any $s \in \R$, the function:
\[
u^{\perp} \times \R \ni (y,t) \mapsto |P_{(a_1 (s u + y) + x_1) \wedge \ldots \wedge (a_k (s u + y) + x_k)} K(t)| 
\]
is jointly convex. The proof is based on the following property, which generalizes Lemma \ref{lem:convex}: for any $x_1,\ldots,x_k \in \R^n$ and $a_1,\ldots,a_k \in \R$, the function $\R^n \ni z \mapsto |P_{(x_1+a_1 z) \wedge \ldots \wedge (x_k+ a_k z)} K|$ is convex. Since we do not require this level of generality here, we leave this to the interested reader. 
\end{rem}

\subsection{Proof of the isoperimetric inequality}

The conclusion of the proof of the inequality (\ref{eq:main}) of Conjecture \ref{conj:Lutwak} is now standard. It is well known  that given a compact set $K \subset \R^n$, there exists a sequence of directions $\{u_i \}_{i=1,2,\ldots} \subset \S^{n-1}$ so that the compact sets $K_i := S_{u_i} S_{u_{i-1}} \ldots S_{u_1} K$ converge in the Hausdorff topology to $B_K$, the Euclidean ball having the same volume as $K$ (see e.g. \cite[Theorem 9.1]{Gruber-ConvexAndDiscreteGeometry} for the case that $K$ is convex or \cite[Lemma 9.4.3]{BuragoZalgallerBook} for the general case). When $K$ is in addition a convex body, since Steiner symmetrization preserves convexity, all the $K_i$ are convex bodies as well. Clearly $\Phi_k$ is continuous on the class of convex bodies with respect to the Hausdorff topology (e.g. \cite[Theorem 1.8.20]{Schneider-Book-2ndEd}), and hence by Theorem \ref{thm:Steiner}:
\[
\Phi_k(K) \geq \Phi_k(K_1) \geq \ldots \geq \Phi_k(K_i) \searrow \Phi_k(B_K)  .  \makebox[0pt][l]{\qquad\qquad\qquad\qquad\qed} 
\]

\section{Further convexity properties} \label{sec:convexity}

After having proved Proposition \ref{prop:key}, we observed that we may actually obtain a different proof 
  of Proposition \ref{prop:key} which is modeled after the Meyer--Pajor proof of the Blaschke--Santal\'o inequality from \cite{MeyerPajor-Santalo}.  While a proof of the more general claim asserted in Remark \ref{rem:more-general} seems to be out of reach of this alternative approach, the proof we will present below has several advantages nevertheless. It highlights an intimate relation between Theorem \ref{thm:main} for general $k$ and the particular case $k=1$ corresponding to the Blaschke--Santal\'o inequality, which in a sense underlies our proof. Most importantly, it reveals a certain additional convexity property of $|P_{E \wedge (y + su)} K(t)|$ in $s$ when $s$ is varied harmonically, which will be crucially used in the characterization of local minimizers of $\Phi_k$. This property was first observed by Meyer--Reisner \cite{MeyerReisner-SantaloViaShadowSystems} for sections of the polar body (in fact, with respect to the Santal\'o point), corresponding to the case $k=1$.

\subsection{A harmonic generalization of Proposition \ref{prop:key}}

\begin{proposition} \label{prop:key-s}
Let $\{K(t)\}_{t \in \R}$ be a shadow system in the direction of $u \in \S^{n-1}$. Given $\alpha \in (0,1)$ and $s_0,s_1 \in \R_+$, set:
\begin{equation} \label{eq:s-lambda}
 \frac{1}{s_{\alpha}} := \frac{1-\alpha}{s_0} + \frac{\alpha}{s_1} ~,~ \lambda = \lambda_\alpha(s_0,s_1) := \frac{\alpha s_0}{\alpha s_0 + (1-\alpha) s_1} .
\end{equation}
Then for all $y_0,y_1 \in u^{\perp}$ and $t_0,t_1 \in \R$, denoting:
\[
y_\lambda := (1-\lambda) y_0 + \lambda y_1 ~,~ t_\alpha = (1-\alpha) t_0 + \alpha t_1 ,
\]
the following holds for any $E \in G_{u^{\perp},k-1}$:
\begin{equation} \label{eq:P-harmonic}
|P_{E \wedge (y_\lambda +  s_{\alpha} u)} K(t_\alpha)| \leq (1-\lambda) |P_{E \wedge (y_0 + s_0 u)} K(t_0)| + \lambda |P_{E \wedge (y_1 + s_1 u)} K(t_1)| . 
\end{equation}
Consequently:
\begin{equation} \label{eq:key-harmonic}
L_{E,u,s_{\alpha}}(K(t_\alpha))\supseteq (1-\lambda) L_{E,u,s_0}(K(t_0)) +  \lambda L_{E,u,s_1}(K(t_1)) ,
\end{equation}
and:
\begin{equation} \label{eq:fgh}
|L_{E,u,s_{\alpha}}(K(t_\alpha))| \geq |L_{E,u,s_0}(K(t_0))|^{1-\lambda} |L_{E,u,s_1}(K(t_1))|^{\lambda} . 
\end{equation}
\end{proposition}

\begin{rem} \label{rem:convention}
Throughout this section, we use the conventions that $s_\alpha = 0$ if $s_0 s_1 = 0$, and $\lambda = \alpha$ if $s_0 = s_1 = 0$. 
\end{rem}

Setting $s_0 = s_1 = s \in \R_+$ (which implies $\lambda = \alpha$), the joint convexity of $u^{\perp} \times \R \ni (y,t) \mapsto|P_{E \wedge (y + su)} K(t)|$ asserted in Proposition \ref{prop:key} immediately follows (the case of general $s \in \R$ is obtained after flipping the signs of $u$ and $t$). In fact, for the purpose of obtaining an alternative proof of Theorem \ref{thm:Steiner}, the only case of interest is when in addition $\alpha=\lambda = 1/2$, $(t_0,t_1) = (1,-1)$ and $\{K(t)\}$ is the linear reflection shadow system $\{K_u(t)\}$ from Lemma \ref{lem:reflection-shadow}. In this case, (\ref{eq:key-harmonic}) reduces to the inequality (\ref{eq:key-inq}) used in the proof of Theorem \ref{thm:Steiner}, after noting that $K_u(t_{1/2}) = K_u(0) = S_u K$ and $L_{E,u,s}(K_u(-1)) = L_{E,u,s}(R_u K) = -L_{E,u,s}(K)$ by (\ref{eq:flip}).
The reader only interested in the proof of Theorem \ref{thm:Steiner} may wish to only consider this case below, but the general case comes at no extra cost and will be vitally used towards the proof of Theorem \ref{thm:intro-local}.

\begin{proof}[Proof of Proposition \ref{prop:key-s}]
Since $|P_{E \wedge x} K|$ only depends on $P_{E^{\perp}} x$, it is enough to establish (\ref{eq:P-harmonic}) for $y_0,y_1 \in E^{\perp} \cap u^{\perp}$. 
Recall the notation $K^w = (K - w) \cap E^{\perp}$ for $w \in P_E K$. If the shadow system is given by $K(t) = T_t^u(\tilde K)$ for some $\tilde K \in \K(\R^{n+1})$, we set $\tilde K^w := (\tilde K - w) \cap \tilde E^\perp$, where $\tilde E^\perp$ denotes the orthogonal complement to $E$ in $\R^{n+1}$. Denoting $K^w(t) := T_t^u(\tilde K^w)$ for $w \in P_E K$, it follows since $u \in E^{\perp}$ that $K(t)^w = K^w(t)$. 
Consequently, to establish (\ref{eq:P-harmonic}), it is enough by (\ref{eq:Fubini}) to verify:
\[
|P_{y_\lambda + s_\alpha u} K^w(t_\alpha) | \leq (1-\lambda) |P_{y_0 + s_0 u} K^w(t_0)| + \lambda |P_{y_1 + s_1 u} K^w(t_1)|  
\]
for all $w \in P_E K$. This is particularly convenient since all projections are one-dimensional intervals. Parametrizing $E^{\perp}$ as $\{ (a,b) := a + b u \; ; \; a \in E^{\perp} \cap u^{\perp} , b \in \R\}$, we verify this as follows (using $* \in \{+,-\}$ and $i \in \{0,1\}$):
\begin{align}
\nonumber &|P_{y_\lambda + s_\alpha u} K^w(t_\alpha) | \\
= & \max \{ \sscalar{y_\lambda , a_{+} - a_{-}} + s_\alpha (b_+ - b_-) \; ; \; (a_*,b_*) \in K^w(t_\alpha)  \} \\
\label{eq:eq-inq} \leq & \max \set{ \sscalar{y_\lambda , a_{+} - a_{-}}  + s_\alpha \brac{\begin{array}{c} ((1-\alpha) r^+_0 + \alpha r^+_1) \\ - ((1-\alpha) r^-_0 + \alpha r^-_1) \end{array} } \; ; \; (a_*,r^{*}_i) \in K^w(t_i) } \\
\label{eq:eq-inq2} = & \max \set{ \sscalar{(1-\lambda) y_0 + \lambda y_1 , a_{+} - a_{-}}  + \brac{\begin{array}{c}
(1-\lambda) s_0 (r^+_0 - r^-_0 ) \\  + \lambda s_1 (r^+_1 - r^-_1 ) \end{array} } \; ; \; (a_*,r^{*}_i) \in K^w(t_i) } \\
\nonumber  \leq & (1-\lambda) \max \{ \sscalar{y_0 , a_+-a_-} + s_0 (r^+_0 - r^-_0) \; ; \; (a_*,r^{*}_0) \in K^w(t_0) \} \\
\nonumber & \;\;\; \; \; + \lambda \max \{ \sscalar{y_1 , a_+-a_-} + s_1 (r^+_1 - r^-_1) \; ; \; (a_*,r^{*}_1) \in K^w(t_1) \} \\
\nonumber  = & (1-\lambda) |P_{y_0 + s_0 u} K^w(t_0)| + \lambda |P_{y_1 + s_1 u} K^w(t_1)| . 
\end{align}
Here  the inequality (\ref{eq:eq-inq}) follows from (\ref{eq:shadow-system-convex}), and the transition in (\ref{eq:eq-inq2}) is due to:
\[
(1-\alpha) s_\alpha = (1-\lambda) s_0 ~,~ \alpha s_\alpha = \lambda s_1 . 
\]
Note that the above identities are consistent with our convention from Remark \ref{rem:convention}. 
The inclusion (\ref{eq:key-harmonic}) immediately follows from (\ref{eq:P-harmonic}) by definition (\ref{eq:LEus-def}), and (\ref{eq:fgh}) is a consequence of (\ref{eq:key-harmonic}) and the Brunn-Minkowski inequality. 
\end{proof}

\subsection{Convexity of the $s$-moment function}

Now introduce the following function, appearing in (\ref{eq:punch}):
\[ M_k(L_E(K_u(t))) := \brac{\int_{L_E(K_u(t))} \abs{\scalar{x,u}}^{k-1} dx}^{-1/k} =  \brac{\int_{\R} |s|^{k-1} |L_{E,u,s}(K_u(t))| ds}^{-1/k}  .
\] 
\begin{thm}  \label{thm:k-convex}
For all $u \in \S^{n-1}$ and $E \in G_{u^{\perp},k-1}$, the function $\R \ni t \mapsto  M_k(L_E(K_u(t)))$ is convex and even. 
\end{thm}

For the proof, we proceed as in \cite{MeyerReisner-SantaloViaShadowSystems}, and invoke the following harmonic Pr\'ekopa--Leindler-type inequality of K.~Ball \cite[p. 74]{Ball-kdim-sections} (see also \cite[Theorems 1.4.6 and 10.2.10]{AGA-Book-I}):
\begin{thm}[Ball] \label{ref:Ball}
Let $f,g,h : \Real_+ \rightarrow \R_+$ be measurable functions so that for some $\alpha \in (0,1)$ and all $s_0,s_1 \in \R_+$:
\[
h(s_\alpha) \geq f(s_0)^{1-\lambda_\alpha(s_0,s_1)} g(s_1)^{\lambda_\alpha(s_0,s_1)} ,
\]
where $s_\alpha$ and $\lambda_\alpha(s_0,s_1)$ are given by (\ref{eq:s-lambda}). 
Then for all $p > 0$, denoting $I_p(w) := \brac{\int_0^\infty s^{p-1} w(s) ds}^{-1/p}$, we have:
\[
I_p(h) \leq (1-\alpha) I_p(f) +  \alpha I_p(g) . 
\]
\end{thm}

\begin{corollary} \label{cor:k-convex-gen} 
With the same assumptions as in Proposition \ref{prop:key-s}, the function
\[
\R \ni t \mapsto \biggl(\int_{L_E(K(t))} \scalar{x,u}_+^{k-1} dx \biggr) ^{-1/k} =  \biggl( \int_{0}^\infty s^{k-1} |L_{E,u,s}(K(t))| ds \biggr) ^{-1/k} 
\]
is convex. 
\end{corollary}
\begin{proof}
Immediate by an application of Theorem \ref{ref:Ball} to the functions $h(s) := w_{t_\alpha}(s)$, $f(s) := w_{t_0}(s)$ and $g(s) := w_{t_1}(s)$ with $w_t(s) := |L_{E,u,s}(K(t))|$ and $t_\alpha := (1-\alpha) t_0 + \alpha t_1$, after recalling (\ref{eq:fgh}). Note that $\R \ni s \mapsto |L_{E,u,s}(K(t))|$ is measurable by the discussion following (\ref{eq:LEus-def}). 
\end{proof}

In the case $k=1$ and when $\{ K(t) \}$ are all origin-symmetric, Corollary \ref{cor:k-convex-gen} amounts to the convexity of $t \mapsto |K(t)^{\circ}|^{-1}$, first established by Campi--Gronchi \cite{CampiGronchi-VolumeProductInqs}, and extended by Meyer--Reisner to the convexity of $t \mapsto |K(t)^{\circ,s}|^{-1}$ for general convex bodies \cite{MeyerReisner-SantaloViaShadowSystems}.  

\bigskip

Applying Corollary \ref{cor:k-convex-gen} to the linear reflection shadow system $\{K_u(t)\}$ from Lemma \ref{lem:reflection-shadow}, for which $K_u(-t) = R_u K_u(t)$ and hence by (\ref{eq:LEus-def}) and (\ref{eq:flip}):
\[
L_{E,u,s}(K_u(-t)) = L_{E,u,-s}(K_u(t)) = - L_{E,u,s}(K_u(t)) =  - L_{E,u,-s}(K_u(-t)) ,
\]
Theorem \ref{thm:k-convex} is established.

\subsection{A dichotomy for $t \mapsto \Phi_k(K_u(t))$} 

By Theorem \ref{thm:monotone} we already know that $\R_+ \ni t \mapsto \Phi_k(K_u(t)) = \Phi_k(K_u(-t))$ is monotone non-decreasing. The next theorem, in which the above convexity will be crucially used, adds vital information -- this function transitions from being constant on $[0,a]$ to strictly monotone on $[a,\infty)$ at a unique $a \in [0,\infty]$.  

\begin{thm} \label{thm:all-t}
Given $u \in \S^{n-1}$ and $t_0 \in \R$,  the equality $\Phi_k(K_u(t_1)) = \Phi_k(K_u(t_0))$ holds for some $|t_1| < |t_0|$ if and only if it holds for all $|t_1| < |t_0|$. 
\end{thm}
\begin{proof}
Assume $\Phi_k(K_u(t_1)) = \Phi_k(K_u(t_0))$ for some $|t_1| < |t_0|$, or equivalently (by Lemma \ref{lem:mu}), $\mu_{k,u}(L_{k,u}(K_u(t_1))) = \mu_{k,u}(L_{k,u}(K_u(t_0)))$. Recall from (\ref{eq:punch}) and the subsequent discussion that:
\begin{equation} \label{eq:mu-Mk}
\mu_{k,u}(L_{k,u}(K_u(t))) = \int_{G_{u^{\perp},k-1}} M_k(L_E(K_u(t)))^{-k} \sigma_{u^{\perp},k-1} (d E) .
\end{equation}
Theorem \ref{thm:k-convex} implies in particular that $\R_+ \ni t \mapsto M_k(L_E(K_u(t))) = M_k(L_E(K_u(-t)))$ is monotone non-decreasing  for all $E \in G_{u^{\perp},k-1}$ (alternatively, a simpler way to deduce the monotonicity is by (\ref{eq:VolLmonotone})). 
It follows that necessarily:
\[
M_k(L_E(K_u(\pm t_1))) = M_k(L_E(K_u(\pm t_0)))  ,
\]
for $\sigma_{u^\perp,k-1}$-almost-every $E \in G_{u^{\perp},k-1}$. Invoking Theorem \ref{thm:k-convex}, it follows that $[-|t_0|,|t_0|] \ni t \mapsto M_k(L_E(K_u(t)))$ must be constant for $\sigma_{u^\perp,k-1}$-almost-every $E \in G_{u^{\perp},k-1}$. Recalling (\ref{eq:mu-Mk}) and Lemma \ref{lem:mu}, we deduce that $\Phi_k(K_u(t)) = \Phi_k(K_u(t_0))$ for all $t \in [-|t_0|,|t_0|]$.
\end{proof}

\begin{rem}
One might try to prove Theorem \ref{thm:all-t} by expanding on the argument of Theorem \ref{thm:monotone}. By Brunn's concavity principle, we know that $\R \ni t \mapsto |L_{E,u,s}(K_u(t))|^{\frac{1}{n-k}}$ is concave and even on its support, but the problem is that for a given $s \in \R$, the support may be a strict subset of $[-|t_1|,|t_1|]$, and so we cannot conclude that the latter function is constant on any interval. 
\end{rem}

The usefulness of a statement like Theorem \ref{thm:all-t} for characterizing \emph{local} extremizers was observed in the case $k=1$ by Meyer--Reisner \cite{MeyerReisner-LocalSantalo}.

\section{Analysis of equality} \label{sec:equality1}

Let $K$ be a convex body in $\R^n$ and fix $k \in \{ 1,\ldots,n-1 \}$. In the next two sections we will establish the equality case of Theorem \ref{thm:intro-Steiner}:
\begin{thm} \label{thm:equality} 
$\Phi_k(K) = \Phi_k(S_u K)$ for all $u \in \S^{n-1}$ if and only if $K$ is an ellipsoid. 
\end{thm}

\begin{corollary} \label{cor:equality}
$\Phi_k(K) = \Phi_k(B_K)$ if and only if $K$ is an ellipsoid. 
\end{corollary}
\begin{proof}[Proof of Corollary \ref{cor:equality} given Theorem \ref{thm:equality}]
If $K$ is an ellipsoid, the invariance of $\Phi_k$ under volume preserving affine maps implies that $\Phi_k(K) = \Phi_k(B_K)$. Conversely, since we always have:
\[
\Phi_k(K) \geq \Phi_k(S_u K) \geq \Phi_k(B_K)  \;\;\; \forall u \in \S^{n-1} 
\]
by Theorems \ref{thm:Steiner} and (\ref{eq:main}), if $\Phi_k(K) = \Phi_k(B_K)$ then we have equality above, and so $K$ must be an ellipsoid by Theorem \ref{thm:equality}. 
\end{proof}

\begin{proof}[Proof of Theorem \ref{thm:intro-local} given Theorem \ref{thm:equality}]
Assume that $K$ is a local minimizer of $\Phi_k$ among all convex bodies of a given volume with respect to the Hausdorff topology. Recall that $\{K_u(t)\}$ denotes the linear reflection shadow system associated to $K$ in the direction of $u \in \S^{n-1}$ from Lemma \ref{lem:reflection-shadow}. For every $u \in \S^{n-1}$, since $t \mapsto K_u(t)$ is continuous in the Hausdorff topology, we know that there exists $\eps \in (0,1)$ so that $\Phi_k(K_u(1-\eps)) \geq \Phi_k(K)$ (as $|K_u(t)| = |K|$ for all $t \in [-1,1]$). On the other hand, by  Theorem \ref{thm:monotone}, we know that $\Phi_k(K) \geq \Phi_k(K_u(t))$ for all $t \in [-1,1]$, and hence we must have equality at $t = 1-\eps$. Therefore,  Theorem \ref{thm:all-t} implies that we have equality for all $t \in [-1,1]$, and in particular at $t=0$, i.e. $\Phi_k(K) = \Phi_k(S_u K)$. Since this holds for all $u \in \S^{n-1}$, Theorem \ref{thm:equality} implies that $K$ must be an ellipsoid. 
\end{proof}

The trivial direction of Theorem \ref{thm:equality} follows from the well-known fact (e.g. \cite[Lemma 2]{BLM-SteinerSymmetrizations}) that Steiner symmetrization transforms an ellipsoid into an ellipsoid (of the same volume), together with the affine-invariance of $\Phi_k$. 
The proof of the non-trivial direction consists of several steps. Steps 1 and 2 are inspired by the Meyer--Pajor simplification \cite{MeyerPajor-Santalo} of Saint-Raymond's analysis in \cite{SaintRaymond-Santalo} 
 of the equality case in the Blaschke--Santal\'o inequality for \emph{origin-symmetric} convex bodies; however, to treat general convex bodies, we put forward several new observations in Steps 3 and 4 which are new even in the classical case $k=1$.  In Step 5 we conclude the proof in the range $1 \leq k \leq n-2$. The remaining case $k=n-1$ requires more work, which is deferred to the next section. 

\subsection{Step 1 - point of symmetry}

Given a convex body $K \subset \R^n$ and $E \in G_{n,k-1}$, recall Definition \ref{defn:E-polar-body} of $L_E(K)$, 
and that $L_E(K)$ is an origin-symmetric convex body in $E^{\perp}$ by Lemma \ref{lem:convex} and the preceding discussion. It is useful to note that:
\begin{lemma} \label{lem:continuous0}
The map $G_{n,k-1} \ni E \mapsto L_E(K) \in \K(\R^n)$ is continuous in the Hausdorff topology. 
\end{lemma}
\begin{proof}
Recall from (\ref{eq:rolodex-prop}) that $\norm{\theta}_{L_E(K)} = |P_{\sspan(E,\theta)} K|$ for all $\theta \in \S(E^{\perp})$. Consequently, $\frac{1}{R} B_{E^{\perp}} \subseteq L_E(K) \subseteq \frac{1}{r} B_{E^{\perp}}$ uniformly for all $E \in G_{n,k-1}$, with $r = \min_{F \in G_{n,k}} |P_F K| > 0$ and $R = \max_{F \in G_{n,k}} |P_F K|> 0$. Since $G_{n,k-1} \ni E \mapsto \frac{1}{r} B_{E^\perp}$ is continuous in the Hausdorff topology, to establish the asserted continuity, it is therefore enough to show that $G_{n,k} \ni F \mapsto |P_F K|$ is continuous (and hence uniformly continuous by compactness of $G_{n,k}$). While for general compact sets $K$ this may be false, for convex bodies $K$ we have (e.g. by (\ref{eq:Fedotov})):
\[
|P_F K| = \frac{{n \choose k}}{|B_2^{n-k}|} V(K,k ; B_{F^{\perp}} , n-k) \;\;\; \forall F \in G_{n,k} ,
\]
and so the continuity of $G_{n,k} \ni F \mapsto |P_F K|$ follows by the (joint) continuity of mixed volume with respect to the Hausdorff topology. 
\end{proof}

Given $u \in \S^{n-1}$ and  $E \in G_{u^{\perp},k-1}$, recall the definition (\ref{eq:LEus-def}) of $L_{E,u,s}(K)$, 
so that $L_{E,u,s}(K)$ is the section of $L_E(K)$ perpendicular to $u$ at height $s \in \R$:
 \begin{equation} \label{eq:section}
 L_{E,u,s}(K) = (L_E(K) - s u) \cap u^{\perp}  \subset E^{\perp} \cap u^{\perp} .
\end{equation}
In particular, $L_{E,u,s}(K)$ is convex and compact, but it may be empty. Denote:
\[
\GG_u(K)  := \{ (E,s) \in G_{u^{\perp},k-1} \times \R \; ; \; L_{E,u,s}(K) \neq \emptyset \} ,
\]
so that $L_{E,u,s}(K) \in \K(E^{\perp} \cap u^{\perp}) \subset \K(\R^n)$ if and only if $(E,s) \in \GG_u(K)$. 
By Lemma \ref{lem:continuous0} and compactness of each $L_E(K)$, it follows that $\GG_u(K)$ is a closed (in fact compact) set (since if $L_{E,u,s}(K) = \emptyset$ then $s u + u^\perp$ must be at a positive distance from $L_E(K)$). 

\smallskip
The following is another immediate consequence of Lemma \ref{lem:continuous0}:
\begin{corollary} \label{cor:interior}
Let $(E,s) \in G_{u^{\perp},k-1} \times \R$ be such that $\interior L_E(K) \cap (s u + u^{\perp}) \neq \emptyset$. Then $(E,s) \in \interior \GG_u(K)$. 
\end{corollary}

We also have:

\begin{lemma} \label{lem:continuous}
The map $\GG_u(K) \ni (E,s) \mapsto |L_{E,u,s}(K)| \in \R_+$ is continuous. 
\end{lemma}
\begin{proof}
The map $\GG_u(K) \ni (E,s) \mapsto L_{E,u,s}(K) \in \K(\R^n)$ is continuous in the Hausdorff topology by Lemma \ref{lem:continuous0} and (\ref{eq:section}). In addition, the map $G_{u^{\perp},k-1} \ni E \mapsto B_{\sspan(E,u)} \in \K(\R^n)$ is trivially continuous in the Hausdorff topology (recall that $B_F$ denotes the unit Euclidean ball in the subspace $F$). 
Finally, note that since $L_{E,u,s}(K) \in \K(E^{\perp} \cap u^{\perp})$
 and $\dim(E^{\perp} \cap u^{\perp}) = n-k$, we have as in Lemma \ref{lem:continuous0} (e.g. by (\ref{eq:Fedotov})):
\[
 |L_{E,u,s}(K)| = \frac{{n \choose k}  }{|B_2^k|} V(B_{\sspan(E,u)} , k; L_{E,u,s}(K), n-k) .
\]
The asserted continuity of $|L_{E,u,s}(K)|$ on $\GG_u(K)$ now follows from the (joint) continuity of mixed volume on $\K(\R^n)^{\times n}$ with respect to the Hausdorff topology. 
\end{proof}
\begin{corollary}
The map $G_{u^{\perp},k-1} \times \R \ni (E,s) \mapsto |L_{E,u,s}(K)|\in \R_+$ is Borel measurable. 
\end{corollary}
\begin{proof}
It is continuous on the compact set $\GG_u(K)$ and zero outside of it.
\end{proof}

Also recall from (\ref{eq:key-inq}) that for all $s \in \R$:
\begin{equation} \label{eq:BM-supset}
L_{E,u,s}(S_u K) \supseteq \frac{1}{2}(L_{E,u,s}(K) - L_{E,u,s}(K)) ,
\end{equation}
and hence  by the Brunn--Minkowski inequality:
\begin{equation} \label{eq:LEus-geq}
|L_{E,u,s}(S_u K)| \geq |L_{E,u,s}(K)| .
\end{equation}
In particular, $\GG_u(K) \subseteq \GG_u(S_u K)$. 

\begin{proposition} \label{prop:yes}
Given $u \in \S^{n-1}$, assume that $\Phi_k(K) = \Phi_k(S_u K)$. Then:
\begin{equation} \label{eq:LEus-equal}
|L_{E,u,s}(S_u K)| = |L_{E,u,s}(K)| 
\end{equation}
for all $E \in G_{u^{\perp},k-1}$ and $s \in \R$, and moreover:
\begin{equation} \label{eq:translation}
L_{E,u,s}(K) = L_{E,u,s}(S_u K) + \alpha_{E,u,s} ,
\end{equation}
for some translation vector $\alpha_{E,u,s} \in E^{\perp} \cap u^{\perp}$.
\end{proposition}
\begin{proof}
By Lemma \ref{lem:mu} and Remark \ref{rem:finite}, the assumption is equivalent to $\mu_{k,u}(L_{k,u}(K)) = \mu_{k,u}(L_{k,u}(S_u K)) < \infty$. In view of (\ref{eq:LEus-geq}) and (\ref{eq:punch}), it follows that (\ref{eq:LEus-equal}) must hold for $\sigma_{u^{\perp},k-1} \otimes \L$-almost-every $(E,s) \in G_{u^{\perp},k-1} \times \R$. By Corollary \ref{cor:interior}, Lemma \ref{lem:continuous} (applied to $K$ and $S_u K$) and the inclusion $\GG_u(K) \subseteq \GG_u(S_u K)$, it follows that (\ref{eq:LEus-equal}) holds for all $(E,s) \in G_{u^{\perp},k-1} \times \R$ so that $\interior L_E(K) \cap (s u + u^{\perp}) \neq \emptyset$. 

Now fix $E \in G_{u^{\perp},k-1}$, and denote by $\Sigma_E$ the (non-empty) interior of the compact interval $P_{\sspan(u)} L_E(K)$, which we identify with an open interval in $\R$ (via $s u \leftrightarrow s$). We know that (\ref{eq:LEus-equal}) holds for all $s \in \Sigma_E$, and that both sides of (\ref{eq:LEus-equal}) are continuous functions of $s$ on their support by (\ref{eq:section}); consequently, the validity of (\ref{eq:LEus-equal}) extends to all $s \in \overline{\Sigma}_E$, the closure of $\Sigma_E$.

 For $s \notin \overline{\Sigma}_E$, $L_E(K) \cap (s u + u^{\perp}) = \emptyset$ and $|L_{E,u,s}(K)| = 0$, and it follows that for $\sigma_{u^{\perp},k-1}$-almost-every $E \in G_{u^{\perp},k-1}$, $|L_{E,u,s}(S_u K)| = 0$ for all $s \notin \overline{\Sigma}_E$, i.e. that $L_E(S_u K) \cap (s u + u^{\perp}) = \emptyset$ for such $s$. In other words, for $\sigma_{u^{\perp},k-1}$-almost-every $E \in G_{u^{\perp},k-1}$:
\begin{equation} \label{eq:gotcha}
L_E(S_u K) \subset \overline{\Sigma}_E \, u + u^{\perp}.
\end{equation}
 But according to Lemma \ref{lem:continuous0}, $G_{u^{\perp},k-1} \ni E \mapsto \overline{\Sigma}_E$ and $G_{u^{\perp},k-1} \ni E \mapsto L_E(S_u K)$ are continuous in the Hausdorff topology, and consequently (\ref{eq:gotcha}) remains valid for all $E \in G_{u^{\perp},k-1}$. In other words, $|L_{E,u,s}(S_u K)| = 0 = |L_{E,u,s}(K)|$ for all $E \in G_{u^{\perp},k-1}$ and $s \notin \overline{\Sigma}_E$, and so we have verified (\ref{eq:LEus-equal}) for all $E \in G_{u^{\perp},k-1}$ and $s \in \R$.  
 
By the equality case of the Brunn-Minkowski inequality, we deduce from (\ref{eq:BM-supset}) and (\ref{eq:LEus-equal}) that for all $E \in G_{u^{\perp},k-1}$ and $s \in \Sigma_E$, equality (\ref{eq:translation}) must hold with some $\alpha_{E,u,s} \in E^{\perp} \cap u^{\perp}$. On the other hand, for $s \notin \overline{\Sigma}_E$ we've already established that $L_{E,u,s}(K) = L_{E,u,s}(S_u K) = \emptyset$, and so (\ref{eq:translation}) also holds for those $s$'s as well. It follows by (\ref{eq:section}) and continuity of the section of a convex body in the Hausdorff topology that (\ref{eq:translation}) remains valid also for $s_0 \in \partial \Sigma_E$ with $\alpha_{E,u,s_0} = \lim_{\Sigma_E \ni s \rightarrow s_0} \alpha_{E,u,s}$. This concludes the proof. 
\end{proof}

Now assume that $\Phi_k(K) = \Phi_k(S_u K)$ for all $u \in \S^{n-1}$.
Recall from the proof of Theorem \ref{thm:Steiner} that $L_{E,u,s}(S_u K)$ is origin-symmetric (being the level-set of the even function $f^{(s)}(\cdot,0)$, as verified in (\ref{eq:flip})). It follows from (\ref{eq:translation}) that $L_{E,u,s}(K)$ has a point of symmetry for all $E \in G_{n,k-1}$, $u \in \S(E^{\perp})$ and $s \in \R$.

 \subsection{Step 2 - Brunn's characterization} 
 
We now invoke the following characterization of ellipsoids, originating in Brunn's 1889 Habilitation \cite{Brunn-Habilitation} (see the False Center Theorem of Aitchison--Petty--Rogers \cite{AitchisonPettyRogers} for a self-contained proof of a more general statement, and Burton--Mani \cite{BurtonMani} for an alternative proof of the latter):
 
 \begin{thm}[Brunn's characterization] \label{thm:Brunn}
 Let $L$ be a convex body in $\R^q$, $q \geq 3$, and let $2 \leq p \leq q-1$. Then $L$ is an ellipsoid iff every $p$-dimensional section of $L$ through its interior has a point of symmetry. 
 \end{thm}
 \begin{rem} \label{rem:Brunn}
 The case when $L$ is a regular convex body in $\R^3$ and $p=2$ is due to Brunn \cite[Chapter IV]{Brunn-Habilitation} (see \cite[Section 4]{Soltan-EllipsoidsSurvey}). In $\R^q$ for general $q \geq 3$, note that it is enough to establish the theorem for $p=2$. Also note that by reverse induction on $p$ and the fact that an ellipsoid has a point of symmetry, it is also enough to establish the case $p=q-1$. 
 According to \cite[Theorem 4.3]{Soltan-EllipsoidsSurvey}, it was shown by Olovjanishnikov \cite{Olovjanishnikov} that it is enough to restrict to hyperplane sections which divide the volume of $L$ in a given ratio $\lambda \neq 1$. A far reaching generalization was obtained by Aitchison--Petty--Rogers \cite[Theorem 2]{AitchisonPettyRogers}, who showed that it is enough to consider all two-dimensional sections which pass through a fixed point $x_0$ in the interior of $L$ which is not a point of symmetry of $L$, if it has one. An alternative proof of the latter characterization was obtained by Burton--Mani \cite{BurtonMani}. 
We refer to the surveys \cite{Soltan-EllipsoidsSurvey, Petty-EllipsoidsSurvey,GruberHobinger} for additional extensions and characterizations of ellipsoids. 
 \end{rem}

Fix $E \in G_{n,k-1}$. Recalling (\ref{eq:section}) and that every $L_{E,u,s}(K)$ has a point of symmetry for all $u \in \S(E^{\perp})$ and $s \in \R$, it
 follows when $\dim E^{\perp} = n-k+1 \geq 3$, i.e. $k \leq n-2$, that $L_E(K)$ is necessarily an ellipsoid in $E^{\perp}$. We proceed assuming this is the case, and defer treating the case $k = n-1$ to the next section. 

\medskip

Note that we are still far from concluding that $K$ is an ellipsoid even in the case $k=1$ (when $E = \{0\}$ and $L_E(K) = (K-K)^{\circ}$), since we have only shown that $K-K$ is an ellipsoid, which does not mean that $K$ itself is an ellipsoid (but rather an affine image of a convex body of constant width). 
 
\subsection{Step 3 - distinguished orthonormal basis}
 
By Lemma \ref{lem:wedge}, for any $A \in \K(\R^n)$ and invertible linear map $T$
which acts invariantly on $E$: 
\[
L_E(T(A)) = \{ x \in E^{\perp} \; ; \; |P_{E \wedge x} T(A)| \leq 1 \} = \{ x \in E^{\perp} \; ; \; |P_{E \wedge T^*(x)} A| \leq 1\} = 
T^{-*}( L_E(A) ).
\]
 
 Since $L_E(K)$ is an (origin-symmetric) ellipsoid in $E^{\perp}$, we may find a positive-definite linear map $T_E$ on $\R^n$ so that $T_E$ acts as the identity on $E$, and on $E^{\perp}$ it maps the Euclidean ball $B_{E^{\perp}}$ onto $L_E(K)$. Denoting:
 \[
 K_E := T_E K ,
 \]
it follows that:
\begin{equation} \label{eq:T1}
 L_E(K_E) = T^{-*}_E(L_E(K)) = B_{E^{\perp}} . 
\end{equation}
 
Let $\{u_i\}_{i=1,\ldots,n-k+1}$ denote an orthonormal basis of $E^{\perp}$ consisting of eigenvectors of $T_E$. As $T_E$ acts diagonally in this basis, observe that the actions of  $S_{u_i}$ and $T_E$ commute. Hence:
\begin{equation} \label{eq:T2}
L_E(S_{u_i} K_E) = L_E(S_{u_i} T_E(K)) = L_E(T_E(S_{u_i} K)) = T^{-*}_E(L_E(S_{u_i} K)) . 
\end{equation}

\medskip

Now recall by (\ref{eq:translation}) and  (\ref{eq:section}) that:
\[
\forall s \in \R \;\; \exists \alpha_{E,u_i,s} \in u_i^{\perp} \;\;  (L_E(K) - su_i) \cap u_i^{\perp} = (L_E(S_{u_i} K) - su_i) \cap u_i^{\perp} + \alpha_{E,u_i,s} .
\]
Applying $T^{-*}_E$ to the last identity, using that it acts invariantly on $\sspan(u_i)$ and $u_i^{\perp}$, and recalling (\ref{eq:T1}) and (\ref{eq:T2}), we deduce:
\[
(B_{E^{\perp}} - su_i) \cap u_i^{\perp} = (L_E(S_{u_i} K_E) - su_i) \cap u^{\perp}_i  + T^{-*}_E(\alpha_{E,u_i,s}) \;\;\; 
\forall s \in \R \;\;\; \forall i=1,\ldots, n-k+1 .
\]
Since $(L_E(S_u K) - su) \cap u^{\perp}$ is origin-symmetric in $E^{\perp} \cap u^{\perp}$, and this does not change under a linear transformation, we know that $(L_E(S_{u_i} K_E) - su_i) \cap u^{\perp}_i$ is also origin-symmetric in $E^{\perp} \cap u^{\perp}_i$ for all $s \in \R$. Since $(B_{E^{\perp}} - su_i) \cap u_i^{\perp}$ is origin-symmetric as well, we deduce that $T^{-*}_E(\alpha_{E,u_i,s}) = 0$ necessarily. It follows that:
\begin{equation} \label{eq:key}
L_E(S_{u_i} K_E) = L_E(K_E) = B_{E^{\perp}} \;\;\; \forall i=1,\ldots, n-k+1 . 
\end{equation}

\subsection{Step 4 - invariance under reflections}

Recall the notation $K^w := (K - w) \cap E^{\perp}$ for $w \in P_E K$. 

\begin{lemma}
Let $K$ be a convex body in $\R^n$, let $E \in G_{n,k-1}$ ($k=1,\ldots,n-1$), and assume that $L_E(S_u K)= L_E(K)$ for some $u \in E^{\perp}$. Then for every $w \in \interior P_E K$, up to translation in the direction of $u$, it holds that $S_u K^w = K^w$, i.e. $K^w$ is invariant under $R_u$, the reflection about $u^{\perp}$. 
\end{lemma}
\begin{proof}
Given $x \in E^{\perp}$, recall from (\ref{eq:Fubini}) that:
\[
\norm{x}_{L_E(K)} = |P_{E \wedge x} K| = \int_{P_E K} |P_{x} K^w| dw = \int_{P_E K} (h_{K^w}(x) + h_{K^w}(-x)) dw.
\]
We are given that $L_E(S_u K) = L_E(K)$, 
and since $(S_u K)^w = S_u K^w$ for all $w \in P_E K$ (as $u \in E^{\perp}$), we deduce that:
\[
\int_{P_E K} (h_{K^w}(x) + h_{K^w}(-x)) dw = \int_{P_E K} (h_{S_u K^w}(x) + h_{S_u K^w}(-x)) dw . 
\]

Note that $S_u K^w \subseteq \frac{1}{2} (K^w + R_u K^w)$, and hence: 
\begin{equation} \label{eq:1inq}
h_{S_u K^w} \leq \frac{1}{2} (h_{K^w} + h_{R_u K^w}) . 
\end{equation}
Since $h_{R_u K^w}(\xi) = h_{K^w}(R_u \xi)$, it follows that:
\[
\int_{P_E K} (h_{K^w}(\xi) + h_{K^w}(-\xi)) dw \leq \frac{1}{2} \int_{P_E K} (h_{K^w}(\xi) + h_{K^w}(R_u \xi) + h_{K^w}(-\xi) + h_{K^w}(-R_u \xi)) dw.
\]
Applying this to $\xi = \theta$ and $\xi = R_u \theta$ for a given $\theta \in E^{\perp}$, and summing, we obtain:
\begin{align}
\nonumber & \int_{P_E K} (h_{K^w}(\theta) + h_{K^w}(-\theta) + h_{K^w}(R_u \theta) + h_{K^w}(-R_u \theta))  dw  \\
\label{eq:4inqs} & \leq \int_{P_E K} (h_{K^w}(\theta) + h_{K^w}(R_u\theta) + h_{K^w}(-\theta) + h_{K^w}(-R_u \theta)) dw
\end{align}
Since both sides are equal, this means that we must have equality for a.e.~$w \in P_E K$ (and hence, by continuity of the corresponding functions on their support, for all $w \in P_E K$), in
the 4 instances of the inequality (\ref{eq:1inq}) we used in the direction of $\xi \in \{ \theta, -\theta, R_u \theta , - R_u \theta \}$ to derive (\ref{eq:4inqs}). We therefore deduce for all $w \in P_E K$:
\[
h_{S_u K^w}(\xi) = \frac{1}{2} (h_{K^w}(\xi) + h_{R_u K^w}(\xi)) \;\;\; \forall \xi \in \{ \theta, -\theta, R_u \theta , - R_u \theta \} . 
\]
Since $\theta$ was arbitrary, it follows that for all $w \in P_E K$:
\[
S_u K^w = \frac{1}{2} (K^w + R_u K^w) .
\]
But by the Brunn--Minkowski inequality:
\[
|K^w| = |S_u K^w| \geq |K^w|^{\frac{1}{2}} |R_u K^w|^{\frac{1}{2}} = |K^w| ,
\]
and the equality case implies that $R_u K^w$ and $K^w$ are translates whenever $\interior K^w \neq \emptyset$, and in particular, whenever $w \in \interior P_E K$. Since there cannot be any translation perpendicular to $u$, the proof is concluded. \end{proof}

Fix $w \in \interior P_E K$. The lemma and (\ref{eq:key}) imply that up to translating in the direction of $u_i$, we have $R_{u_i} K_E^w = K_E^w$. Since the $u_i$'s are all orthogonal, it follows that there is a single translation of $K_E^w$ so that $R_{u_i} K_E^w = K_E^w$ for all $i=1,\ldots,n-k+1$.  Since the composition of all $R_{u_i}$'s is precisely $ -\text{Id}$ on $E^{\perp}$, we deduce that $K_E^w$ has a point of symmetry. Recalling that $K_E = T_E(K)$ and that $T_E$ acts as the identity on $E$, it follows that $K^w$ has a point of symmetry. 
 
\subsection{Step 5 - concluding when $1 \leq k \leq n-2$}

We have shown that for every $E \in G_{n,k-1}$, every section $K \cap (w + E^{\perp}) = w + K^w$ of $K$ through its interior has a point of symmetry. It follows by Brunn's Theorem \ref{thm:Brunn} that whenever $n \geq 3$ and $\dim E^{\perp} = n-k+1 \geq 2$, i.e. $k \leq n-1$, $K$ must be an ellipsoid. 

\medskip

All in all, this establishes Theorem \ref{thm:equality} when $1 \leq k \leq n-2$ (and hence $n \geq 3$). 
The case when $k=n-1$ will be handled in the next section.

\section{Analysis of equality when $k=n-1$} \label{sec:equality2}

To establish the case $k=n-1$ of Theorem \ref{thm:equality}, we cannot invoke Brunn's Theorem \ref{thm:Brunn} in Step 2 of the previous section, since $\dim E^{\perp} = 2$ for $E \in G_{n,k-1}$. 
In this section we describe a more complicated argument for bypassing Step 2 when $k=n-1$. 

\subsection{Linear boundary segments}

We will need the following two-dimensional observation (compare with \cite[Lemma 8]{MeyerReisner-SantaloViaShadowSystems}, which is insufficient for our purposes). Recall our notation:
\[
L(K) = L_{\{0\}}(K) = (K-K)^{\circ} . 
\]

\begin{proposition} \label{prop:segments}
Let $\{K(t)\}_{t \in \R}$ denote a shadow system of convex bodies in $\R^2$ in the direction of $e_2$. Given two non-empty open intervals $S,T \subset \R$, assume that there exist functions $a, \Psi : S \rightarrow \R$ so that $(a(s) + \Psi(s) t , s) \in \partial L(K(t))$ for all $s \in S$ and $t \in T$. Then there exist $c_+,c_- \in \R$ so that $\Psi(s) = c_+ s_+ - c_- s_-$ for all $s \in S$. 
\end{proposition}
\begin{proof}
By definition, there exists $\tilde K \in \K(\R^3)$ so that $K(t) = T^{e_2}_t(\tilde K)$ where $T^{e_2}_t : \R^3 \rightarrow \R^2$ is a projection onto $\R^2$ parallel to $e_3 + t e_2$. As in the proof of Proposition \ref{prop:key}, we have:
\begin{align*}
\norm{(y,s)}_{L(K(t))} &= h_{K(t)}(y,s) + h_{K(t)}(-y,-s) \\
& = h_{\tilde K}(y,s,-st) + h_{\tilde K}(-y,-s,st) = \norm{(y,s,-st)}_{L(\tilde K)} . 
\end{align*}
Our assumption then yields the following local parametrization of the surface $\partial L(\tilde K)$:
\[
F(s,t) := (a(s) + \Psi(s) t , s , -st) \in \partial L(\tilde K) \;\;\; \forall s \in S \;\; \forall t \in T . 
\]

The convexity of $L(\tilde K)$ implies that its boundary may locally be represented by a convex function $f$, which is therefore Lipschitz and hence differentiable almost-everywhere by Rademacher's theorem. Moreover, by Alexandrov's theorem (e.g. \cite[Chapter 2]{Gruber-ConvexAndDiscreteGeometry}), $f$ is twice-differentiable (in Alexandrov's sense) almost-everywhere. At points of first differentiability, two linearly independent tangent vectors to the boundary are given by:
\[
\partial_s F(s,t) = (a'(s) + \Psi'(s) t , 1 , -t )  ~,~ \partial_t F(s,t) = (\Psi(s) , 0 , -s) ,
\]
and so the normal to the boundary is in the direction:
\[
 N := (s , -s a'(s) - s \Psi'(s) t + t \Psi(s) , \Psi(s)) .
\]
At points of second differentiability, the surface has a second-order Taylor expansion governed by the second fundamental form:
\[
\II :=  \scalar{\partial_s^2 F , N/|N|} ds^2 + 2 \scalar{\partial_s \partial_t F , N/|N|} dt ds + \scalar{\partial_t^2 F , N/|N|} dt^2 .
\]
Since $\partial^2_t F \equiv 0$ and $\scalar{\partial_s \partial_t F , N} = s \Psi'(s) - \Psi(s)$, we see that unless that latter term vanishes, $\II$ will have strictly negative determinant, implying that the surface has a saddle at that point, contradicting convexity.

We now claim that the only (locally) Lipschitz function $\Psi$ which solves $s \Psi'(s) - \Psi(s) = 0$ for almost every $s \in S$ is of the form $\Psi(s) = c_+ s_+ - c_- s_-$. Indeed, denote $S_+$ and $S_- $ the open subsets of $S$ where (the continuous) $\Psi$ is positive and negative, respectively.  On $S_+$ we have $(\log \Psi)'(s) = (\log s)'$ and so by (local) absolute continuity of $\log \Psi$ we deduce that $\Psi(s) = c_{i} s$ ($c_{i} \neq 0$) on each connected component $S_{+,i}$ of $S_+$; since $\Psi$ must vanish at the end-points of each connected component which lie in $S$, this implies that there is at most one connected component in each of $S \cap \R_+$ and $S \cap \R_-$, and that its end-point in $S$ must be at $s=0$. An analogous statement holds on $S_-$. This implies that $\Psi$ must be of the asserted form. 
\end{proof}

\subsection{Step 1 - segments of constant projections of $K$}

Fix $E \in G_{n,k-1}$ and $u \in \S(E^{\perp})$. The argument of Step 1 from the previous section gives us a little more information than was stated there.
Given $s \in \R$, recall the definition of $f^{(s)}$ from the proof of Theorem \ref{thm:Steiner}:
\[
(E^{\perp} \cap u^{\perp}) \times \R \ni (y,t) \mapsto f^{(s)}(y,t) := |P_{E \wedge (y + su)} K_u(t)| .
\]
We know that $f^{(s)}$ is convex and even in $(y,t)$, and hence its level set:
 \[
\tilde L_{E,u,s} := \{(y,t)  \in (E^{\perp} \cap u^{\perp}) \times \R \; ; \; f^{(s)}(y,t) \leq 1\} 
\]
is convex and origin-symmetric. Note that $\tilde L_{E,u,s}(t) = L_{E,u,s}(K_u(t))$, where we denote by $A(t)$ the $t$-section of $A$. 
By Brunn's concavity principle as in the proof of Theorem~\ref{thm:all-t}, it follows that $\R \ni t \mapsto |\tilde L_{E,u,s}(t)|^{\frac{1}{n-k}}$ is even and concave on its support.

If $\Phi_k(K) = \Phi_k(S_u K)$, we know that $|\tilde L_{E,u,s}(1)| = |\tilde L_{E,u,s}(-1)| = |\tilde L_{E,u,s}(0)|$, and so $[-1,1] \ni t \mapsto |\tilde L_{E,u,s}(t)|$ must be constant. Recall from the proof of Proposition \ref{prop:yes} that $\Sigma_E$ denotes the 
(non-empty) interior of the compact interval $P_{\sspan(u)} L_E(K)$, which we identify with an open interval in $\R$ (via $s u \leftrightarrow s$).
By the equality case of the Brunn-Minkowski inequality, we deduce for all  $s \in \Sigma_E$ that $\tilde L_{E,u,s} \cap \{ t \in [-1,1]\}$ must be a tilted cylinder over the origin-symmetric base $\tilde L_{E,u,s}(0) = L_{E,u,s}(S_u K) \subset E^{\perp} \cap u^{\perp}$. In other words, for all $s \in \Sigma_E$:
 \[
L_{E,u,s}(K_u(t)) = L_{E,u,s}(S_u K) + \alpha_{E,u,s} t \;\;\; \forall t \in [-1,1] .
\]
The same argument as in the proof of Proposition \ref{prop:yes} verifies that this extends to all $s \in \R$, thereby generalizing (\ref{eq:translation}). 
 
Now, denote for $R > 0$:
 \[
\tilde L_{E,u,s,R} := \{(y,t)  \in (E^{\perp} \cap u^{\perp}) \times \R \; ; \; f^{(s)}(y,t) \leq R\} .
\]
By homogeneity of $f^{(s)}(y,t)$ in $(y,s)$ and a simple rescaling:
\[
\tilde L_{E,u,s,R}(t) = R \tilde L_{E,u,s/R,1}(t) \;\;\; \forall t , 
\]
and hence:
\[
\tilde L_{E,u,s,R}(t) = R L_{E,u,s/R}(S_u K) + R \alpha_{E,u,s/R} t \;\;\; \forall t \in [-1,1] .
\]
Using evenness, it follows that for every $y \in E^{\perp} \cap u^{\perp}$ so that $f^{(s)}(y,0) = R$, we have:
\[
 \text{$f^{(s)} \equiv R$ on both segments $\{ (\pm y + R \alpha_{E,u,s/R} t,t) \; ; \; t \in [-1,1]\}$}. 
\]
   
\subsection{Step 2 - segments of constant projections of $K^w$}
 
Recall that for every $w \in \interior P_E K$, $K^w = (K - w) \cap E^{\perp}$ has an interior in $E^{\perp}$ (and is therefore a convex body there). Since $\{(K^w)_u(t)\}_{t \in \R}$ are all convex bodies as well, this also ensures that $\{L((K^w)_u(t))\}_{t \in \R}$ are origin-symmetric convex bodies in $E^{\perp}$. Note that $(K^w)_u(t) = (K_u(t))^w$, and so we simply denote this by $K^w_u(t)$. 
For $w \in \interior P_E K$ and $s \in \R$, denote:
\[
(E^{\perp} \cap u^{\perp}) \times \R \ni (y,t) \mapsto f^{(s)}_w(y,t) := |P_{y + su} K^w_u(t)| ,
\]
so that:
\[
L(K^w_u(t)) = \{ y + s u \; ; \; f^{(s)}_w(y,t) \leq 1 \} . 
\]
Recall from (\ref{eq:Fubini}) that:
\begin{equation} \label{eq:fw-rep}
f^{(s)}(y,t) = \int_{P_E K} f^{(s)}_w(y,t) dw ,
\end{equation}
and that each $f^{(s)}_w$ is convex and even in $(y,t)$ by Proposition \ref{prop:key}. 

\medskip

Denote by $\Sigma_{w}(t)$ the (non-empty) interior of the compact interval $P_{\sspan(u)} L(K^w_u(t))$, which we identify with an open interval in $\R$ (via $s u \leftrightarrow s$). We claim that:
\[ \Sigma_{w} := \Sigma_w(0) =  \Sigma_{w}(t)  \;\;\; \forall t \in [-1,1] . 
\] Indeed, since the projection of the polar equals the polar of the section:
\[
P_{\sspan(u)} L(K^w_u(t)) = P_{\sspan(u)} (K^{w}_u(t) - K^{w}_u(t))^{\circ} =  ( (K^{w}_u(t) - K^{w}_u(t)) \cap \sspan(u) )^{\circ} . 
\]
But since $K^w_u(t) = \cup_{y \in P_{u^{\perp}} K^w} (y + (c_w(y) t + [-\ell_w(y),\ell_w(y)]) u)$ for all $t \in [-1,1]$, we have:
\[
(K^{w}_u(t) - K^{w}_u(t)) \cap \sspan(u) = \cup_{y \in P_{u^{\perp}} K^w}  [-2 \ell_w(y),2 \ell_w(y)] u ,
\]
which is independent of $t$. 

\medskip

Now fix $w_0 \in \interior P_E K$. Let $\pm y_s + s u \in \partial L(K^{w_0}_u(0))$ for $s \in \Sigma_{w_0}$, amounting to $f^{(s)}_{w_0}(\pm y_s,0) = 1$. Denote $R_{y_s,s} := f^{(s)}(\pm y_s,0)$. Since $P_E K \ni w \mapsto f^{(s)}_w(x)$ is continuous for every $x=(y,t)$, since $f^{(s)}_w$ are all convex, and since $f^{(s)} \equiv R_{y_s,s}$ on both segments $\{ (\pm y_s + R_{y_s,s} \alpha_{E,u,s/R_{y_s,s}} t , t) \; ; \; t \in [-1,1]\}$ by the previous step, it follows that each $f^{(s)}_w$  must be constant on these two segments as well, and in particular:
  \[
 \text{$f^{(s)}_{w_0} \equiv 1$ on both segments $\{ ( \pm y_s  + R_{y_s,s} \alpha_{E,u,s/R_{y_s,s}} t ,t) \; ; \; t \in [-1,1]\}$}. 
 \]
 By convexity of $f^{(s)}_{w_0}$,
  this implies that for all $s \in \Sigma_{w_0}$ and $\pm y_s + s u \in \partial L(K^{w_0}_u(0))$ we have:
 \[
 \pm y_s  + R_{y_s,s} \alpha_{E,u,s/R_{y_s,s}} t + s u \in \partial L(K^{w_0}_u(t)) \;\;\; \forall t \in [-1,1]  . 
 \]
 
\subsection{Step 3 - using $k=n-1$}

When $k=n-1$ we have $\dim E^{\perp} = 2$, and so the $K^w$'s are two-dimensional convex bodies for all $w \in \interior P_E K$. Given $A \subset E^{\perp}$, let us denote by $A(s)$ the one-dimensional chord $(A - su)\cap (E^{\perp} \cap u^{\perp})$, which we identify with a subset of $\R$. The discussion in Step 2 implies that for all $w \in \interior P_E K$: 
\begin{equation} \label{eq:representation}
L(K^w_u(t))(s) = [-a_w(s),a_w(s)] + \Psi_w(s) t \;\;\; \forall s \in \Sigma_w \;\;\; \forall t \in [-1,1] . 
\end{equation}
It follows by Proposition \ref{prop:segments} that $\Psi_w(s) = c^w_+ s_+ - c^w_- s_-$ for some $c^w_{\pm} \in \R$ and all $s \in \overline{\Sigma}_w = P_{\sspan(u)} L(K^w_u(t))$ (the claim on the open $\Sigma_w$ extends by continuity of the mid-point to the entire closure $\overline{\Sigma}_w$). But origin-symmetry of $L(K^w)$ and the representation (\ref{eq:representation}) for $t=1$ implies that $\Psi_w$ must be odd, and hence $c^w := c^w_+ = c^w_-$. We deduce that the mid-point of the chord of $L(K^w)$ perpendicular to $u$ at height $s$ is $c^w s$ for all values of $s$ for which the chord is non-empty, and hence all mid-points lie on a single line. This remains true for any $u \in \S(E^{\perp})$, since we assume equality $\Phi_k(K) = \Phi_k(S_u K)$ for all directions $u$. 

\smallskip

We can now invoke the following classical characterization of ellipsoids, due to Bertrand \cite{Bertrand-Ellipsoids} and to Brunn \cite[Chapter IV]{Brunn-Habilitation}; their original statement applied to the plane, but easily extends to $\R^n$. See the historical discussion in \cite[Section 8]{Soltan-EllipsoidsSurvey} and \cite[Theorem 2.12.1]{ConstantWidthBook} or \cite[Theorem 9.2.4]{GardnerGeometricTomography2ndEd} for a proof.
\begin{thm}[Bertrand--Brunn]
Let $K$ be a convex body in $\R^n$. Then $K$ is an ellipsoid if and only if for any direction $u$, the mid-points of all (one-dimensional) chords of $K$ parallel to $u$ lie in a hyperplane. 
\end{thm}

We deduce from the Bertrand--Brunn Theorem that $L(K^w)$ must be an (origin-symmetric) ellipsoid for all $w \in \interior P_E K$. 

\subsection{Step 4 - concluding the proof}

For every $w \in \interior P_E K$ we know that $L(K^w) = T_w(B_{E^{\perp}})$ for some linear map $T_w : E^{\perp} \rightarrow E^{\perp}$ and that 
\[
L(K^w)(s) = L(S_u K^w)(s) + c_w s \;\;\; \forall s .
\]
We may therefore invoke the argument of Step 3 of the previous section to deduce that there exist two orthogonal directions $u_1,u_2 \in \S(E^{\perp})$ so that:
\[
L(S_{u_i} T_w(K^w)) = L(T_w(K^w)) = B_{E^{\perp}} \;\;\; i=1,2 . 
\]
Invoking the argument of Step 4 of the previous section, it follows that $T_w(K^w)$ is invariant (up to translation in the direction $u_i$) under reflection about $u_i^{\perp}$, and hence $T_w(K^w)$ and therefore $K^w$ have a point of symmetry for every $w \in \interior P_E K$. As this holds true for all $E \in G_{n,n-2}$, we deduce that every two-dimensional section of $K$ through its interior has a point of symmetry. It follows as in Step 5 of the previous section that if $n \geq 3$ then $K$ must be an ellipsoid by Brunn's Theorem \ref{thm:Brunn}.

When $n=2$ things are even simpler, since $E = \{0\}$ and so $\interior P_E K = \{0\}$ and $K = K^{w}$ for $w = 0$; we know that $T_0(K) = x_0 + C$ for some origin-symmetric convex body $C$, and that:
\[
(2 C)^{\circ} = (T_0(K) - T_0(K))^{\circ} = L(T_0(K)) = B_2^2,
\]
implying that $C = \frac{1}{2} B_2^2$ and hence that $K$ is an ellipsoid. This concludes the proof of Theorem \ref{thm:equality} when $k=n-1$.

\section{$L^p$-moment quermassintegrals and Alexandrov--Fenchel-type inequalities} \label{sec:AF}

\subsection{$L^p$-moment quermassintegrals}

\begin{defn} \label{def:Q}
Given $k=1,\ldots,n$ and $p \in \R$, denote the $L^p$-moment quermassintegrals of a convex body $K$ in $\R^n$ as:
\[
\Q_{k,p}(K) := \frac{|B_2^n|}{|B_2^k|} \brac{\int_{G_{n,k}} |P_F K|^{p} \sigma_{n,k}(dF)}^{1/p}  . 
\]
 The case $p=0$ is interpreted in the limiting sense as:
\[
\Q_{k,0}(K) = \frac{|B_2^n|}{|B_2^k|}  \exp \brac{ \int_{G_{n,k}} \log |P_F K| \sigma_{n,k}(dF) } . 
\]
\end{defn}

Note that when $p=-n$ we recover the affine quermassintegrals:
\[
\Q_{k,-n}(K) = \Phi_k(K) . 
\]
It is known (see \cite{Grinberg-Affine-Invariant} or Remark \ref{rem:optimal-p} below) that $p=-n$ is the unique value of $p \in \R$ for which $\Q_{k,p}(K)$ is invariant under volume-preserving affine transformations of $K$ (explaining why we prefer to use the notation $\Q_{k,p}$ instead of $\Phi_{k,p}$). 

When $p=1$, we obtain by Kubota's formula (\ref{eq:intro-W2}) the classical quermassintegrals $\Q_{k,1}(K) = W_k(K)$. 
The case $p=-1$ corresponds to the harmonic quermassintegrals $\Q_{k,-1}(K) = \hat W_k(K)$ defined by Hadwiger \cite[Subsection 6.4.8]{Hadwiger-Book} and further studied by Lutwak \cite{Lutwak-Isepiphanic,Lutwak-Harmonic}. 
We will present an interpretation of the case $p=0$ as an averaged version of the Loomis-Whitney inequality in the next subsection. 

\medskip

Having Theorem \ref{thm:main} at hand, we can easily deduce:
\begin{thm} \label{thm:Q}
For any convex body $K$ in $\R^n$, $k=1,\ldots,n-1$ and $p > -n$:
\begin{equation} \label{eq:Q-inq}
\Q_{k,p}(K) \geq \Q_{k,p}(B_K) ,
\end{equation}
with equality iff $K$ is a Euclidean ball. 
\end{thm}

For $p=1$, this is the classical isoperimetric inequality for the quermassintegrals given by Theorem \ref{thm:FAF}. 
For $p=-1$, the above isoperimetric inequality for the harmonic quermassintegrals was obtained by Lutwak in \cite{Lutwak-Harmonic} (cf. \cite[p. 514]{Schneider-Book-2ndEd}). For $-n < p < -1$, this appears to be new. 

\begin{proof}[Proof of Theorem \ref{thm:Q}]
By Jensen's inequality and Theorem \ref{thm:main}:
\[
\Q_{k,p}(K) \geq \Q_{k,-n}(K)  = \Phi_k(K) \geq \Phi_k(B_k) = \Q_{k,-n}(B_K) = \Q_{k,p}(B_K) . 
\]
If $\Q_{k,p}(K) = \Q_{k,p}(B_K)$ then we have equality in both inequalities above. Equality in the second implies by Theorem \ref{thm:equality} that $K$ is an ellipsoid. Equality in the first (Jensen's inequality) implies that $G_{n,k} \ni F \mapsto|P_F K|$ is constant, and hence $K$ must be a Euclidean ball. 
\end{proof}

\begin{rem} \label{rem:optimal-p}
The value of $p=-n$ is precisely the sharp threshold below which the inequality (\ref{eq:Q-inq}) is no longer true, even for ellipsoids. Indeed, if $p < -n$ and $K$ is any ellipsoid which is not a Euclidean ball, then by Jensen's inequality (which yields a strict inequality since $G_{n,k} \ni E \mapsto |P_E K|$ is not constant) and affine invariance of $\Q_{k,-n}$:
\[
\Q_{k,p}(K) < \Q_{k,-n}(K) = \Q_{k,-n}(B_K) = \Q_{k,p}(B_K) . 
\]
\end{rem}

\subsection{Case $p=0$ - averaged Loomis--Whitney} \label{subsec:LW}

The classical Loomis--Whitney inequality (for sets) \cite{LoomisWhitney} asserts that if $K$ is a compact set in $\R^n$ then:
\begin{equation} \label{eq:LW}
\Pi_{i=1}^{n}|P_{e_i^{\perp}} K| \geq |K|^{n-1} ,
\end{equation}
with equality when $K$ is a box (i.e. a rectangular parallelepiped with facets parallel to the coordinate axes). From this, Loomis and Whitney deduce in \cite{LoomisWhitney} a non-sharp form of the isoperimetric inequality for the surface area $S(K)$ (for any reasonable definition of the latter):
\begin{equation} \label{eq:non-sharp-LW}
S(K) \geq 2 |K|^{\frac{n-1}{n}} . 
\end{equation}

Let $I_k := \{ I \subseteq \{1,\ldots,n\} \; ; \; |I|=k \}$, and given $I \in I_k$, denote $E_I = \sspan\{ e_i \; ; \; i \in I \}$. By reverse induction on $k$, it is easy to deduce the following extension of (\ref{eq:LW}) to any $k=1,\ldots,n-1$:
\[
\Pi_{I \in I_k} |P_{E_I} K| \geq |K|^{{n-1 \choose k-1} } ,
\]
with equality when $K$ is a box. In the class of convex bodies, this is also a necessary condition for equality. See \cite{Ball-Shadows,Zhang-AffineSobolevInq,BollobasThomason-Cubes} for further extensions.  

Of course, one can choose any other orthonormal basis in the Loomis--Whitney inequality instead of $\{e_1,\ldots,e_n\}$, and it is a natural question to ask whether a better inequality holds if we average over all possible orthonormal bases (as a cube is not invariant under rotations). Taking a geometric average gives a particularly pleasing result, since:
\begin{align*}
\int_{\SO(n)} \log  \Pi_{I \in I_k} |P_{U(E_I)} K| \sigma_{\SO(n)}(dU) & = {n \choose k} \int_{G_{n,k}} \log |P_E K|\sigma_{n,k}(dE) \\
& = {n \choose k}  \log  \brac{\frac{|B_2^k|}{|B_2^n|} \Q_{k,0}(K)} .
\end{align*}
Consequently, we have by Theorem \ref{thm:Q} (in fact, this already follows from Lutwak's confirmation in \cite{Lutwak-Harmonic} of the case $p=-1$), that for any convex body $K$ in $\R^n$ and $k=1,\ldots,n-1$:
\[ \exp \brac{\int_{\SO(n)} \log  \Pi_{I \in I_k} |P_{U(E_I)} K| \sigma_{\SO(n)}(dU) }\geq \brac{\frac{|B_2^k|}{|B_2^n|} \Q_{k,0}(B_K)}^{{n \choose k}} 
= |B_2^k|^{{n \choose k}} \brac{\frac{|K|}{|B_2^n|}}^{{n-1 \choose k-1} } ,
\] with equality if and only if $K$ is a Euclidean ball. This means that for the averaged Loomis--Whitney inequality, it is not the cube which is optimal but rather the Euclidean ball. Moreover, since by Jensen's inequality and Theorem \ref{thm:Q}:
\[
 \frac{1}{n} S(K) = W_{n-1}(K) = \Q_{n-1,1}(K) \geq \Q_{n-1,0}(K) \geq \Q_{n-1,0}(B_K) = |B_2^n| \brac{\frac{|K|}{|B_2^n|}}^{\frac{n-1}{n}} ,
\]
this averaged version implies the classical isoperimetric inequality (for convex bodies) with a \emph{sharp} constant, in contrast to the non-sharp (\ref{eq:non-sharp-LW}). 

\subsection{Alexandrov--Fenchel-type inequalities} \label{subsec:AF}

It was noted by Lutwak in \cite{Lutwak-Isepiphanic}, following Hadwiger \cite[Subsection 6.4.8, Satz XV]{Hadwiger-Book} for the case $p=-1$, that:
\[
\Q^{1/k}_{k,p}(K_1 + K_2) \geq \Q^{1/k}_{k,p}(K_1) + \Q^{1/k}_{k,p}(K_2)  \;\;\; \forall k p \leq 1 . 
\]
In particular, this holds for all $k=1,\ldots,n-1$ when $p \leq 0$. Indeed, this follows from the classical Brunn--Minkowski inequality for $P_F (K_1 + K_2) = P_F K_1 + P_F K_2$ and the  reverse triangle inequality for the $L^{kp}$-norm when $kp \leq 1$. A nice feature is that this extends to all compact sets $K_1,K_2$ (by Lusternik's extension of the Brunn--Minkowski inequality to compact sets \cite[Section 8]{BuragoZalgallerBook}). 

This suggests that perhaps there is some Brunn--Minkowski-type theory for the $L^p$-moment quermassintegrals when $p \leq 0$, and of particular interest is the affine-invariant case $p=-n$. 

\medskip

It will be more convenient to use the following normalization, already defined in the Introduction:
\[
\I_{k,p}(K) := \frac{\Q_{k,p}(K)}{\Q_{k,p}(B_K)} = \brac{\frac{\int_{G_{n,k}} |P_F K|^{p} \sigma_{n,k}(dF)}{\int_{G_{n,k}} |P_F B_K|^{p} \sigma_{n,k}(dF)}}^{1/p} . 
\]
Note that $\I_{k,p}(B) = 1$ for any Euclidean ball $B$ and all $k,p$, that $\I_{n,p}(K) = 1$ for all $p$, that by Jensen's inequality:
\[
[-n,1] \ni p \mapsto \I_{k,p}(K)  \text{ is non-decreasing} ,
\]
and that Theorems \ref{thm:main} and \ref{thm:Q} imply:
\begin{equation} \label{eq:Igeq1}
\I_{k,p}(K) \geq 1 \;\;\; \forall p \geq -n ,
\end{equation}
with equality for $p > -n$ ($p=-n$) if and only if $K$ is a Euclidean ball (ellipsoid). 

\smallskip

In the classical case $p=1$, Alexandrov's inequalities \cite[(7.67)]{Schneider-Book-2ndEd}, \cite[Subsection 20.2]{BuragoZalgallerBook} (a particular case of the Alexandrov--Fenchel inequalities) assert that:
\[
\I_{1,1}(K) \geq \I^{1/2}_{2,1}(K) \geq \ldots \geq \I^{1/k}_{k,1}(K) \geq \ldots \geq \I^{1/(n-1)}_{n-1,1}(K) \geq \I^{1/n}_{n,1}(K) = 1 .
\]

In view of all of the above, the following was proved by Lutwak for $p=-1$ and conjectured to hold for $p=-n$ in \cite{Lutwak-Harmonic}
(see also \cite[Problem 9.5]{GardnerGeometricTomography2ndEd}):

\begin{conj} \label{conj:sec-AF}
For all $p \in [-n,0]$ and for any convex body $K$ in $\R^n$:
\[
\I_{1,p}(K) \geq \I^{1/2}_{2,p}(K) \geq \ldots \geq \I^{1/k}_{k,p}(K) \geq \ldots \geq \I^{1/(n-1)}_{n-1,p}(K) \geq \I^{1/n}_{n,p}(K) = 1 .
\]
\end{conj}

Our isoperimetric inequality (\ref{eq:Igeq1}) establishes the inequality between each of the terms and the last one. Theorem \ref{thm:intro-half-conj} from the Introduction, which we repeat here for convenience, confirms ``half" of the above conjecture. 

\begin{thm} \label{thm:half-conj}
For every $p \in [-n,0]$ and $1 \leq k \leq m \leq n$:
\[
\I^{1/k}_{k,p}(K) \geq \I^{1/m}_{m,p}(K) ,
\]
for any convex body $K$ in $\R^n$ whenever $m \geq -p$. When $k < m < n$, equality holds for $p \geq -m$ if and only if $K$ is a Euclidean ball. When $k < m = n$, equality holds for $p > -n$ ($p=-n$) if and only if $K$ is a Euclidean ball (ellipsoid). 
\end{thm}
This confirms the conjecture for all $p \in [-2,0]$ and recovers the case $p=-1$ established by Lutwak in \cite{Lutwak-Harmonic}. Our argument is very similar to the one used by  Lutwak; however, instead of relying on Petty's projection inequality, we have the full strength of Theorem \ref{thm:main} at our disposal, which is crucial for handling the range $p < -(k+1)$. 

\medskip

As remarked in the Introduction, the analogous statement for the dual $L^p$-moment quermassintegrals $\tilde \I_{k,p}$ was established by Gardner \cite[Theorem 7.4]{Gardner-DualAffineQuermassintegrals} in exactly the same corresponding range of parameters as in Theorem \ref{thm:half-conj}. Namely, defining
\[
\tilde \I_{k,p}(K)  := \brac{\frac{\int_{G_{n,k}} |K \cap F|^{p} \sigma_{n,k}(dF)}{\int_{G_{n,k}} |B_K \cap F|^{p} \sigma_{n,k}(dF)}}^{1/p} ,
\]
Gardner proved that $\tilde \I^{1/k}_{k,p}(K) \leq \tilde \I^{1/m}_{m,p}(K)$ for an arbitrary bounded Borel set $K$ and for all $1 \leq k \leq m \leq n$ and $p \in (0,m]$, with precise characterization of the equality conditions. However, he also showed in \cite[Theorem 7.7]{Gardner-DualAffineQuermassintegrals} that an analogous statement to Conjecture \ref{conj:sec-AF} cannot hold for $p=n$, $k=1$ and $2 \leq m \leq n-1$: any origin-symmetric star body $K$ which is not an ellipsoid satisfies $\tilde \I_{1,n}(K) = \tilde \I_{n,n}(K) = 1$, but $\tilde \I_{m,n}(K) < 1$. This troubling sign is likely a special feature of the dual setting when $k=1$, since $\tilde \I_{1,n}(K) = 1$ for all origin-symmetric star bodies, whereas  $\tilde \I_{m,n}(K) = 1$ if and only if $K$ is an origin-symmetric ellipsoid (up to a null-set) when $2 \leq m \leq n-1$ \cite[Corollary 7.5]{Gardner-DualAffineQuermassintegrals}. This abnormal behavior of the case $k=1$ does not occur in the primal setting, thanks to the equality conditions of the Blaschke--Santal\'o inequality. 

\begin{proposition} \label{prop:bootstrap}
For all $1 \leq k \leq  m \leq n$ and $q \leq m$:
\begin{equation} \label{eq:I-move}
\I^{1/k}_{k,-q}(K) \geq \I^{1/m}_{m,-q\frac{k}{m}}(K) . 
\end{equation}
When $k < m$, equality holds for $q < m$ ($q=m$) if and only if $K$ is a Euclidean ball (ellipsoid).
\end{proposition}
\begin{proof}
Applying (\ref{eq:Igeq1}) on $E$ for the convex body $P_E K$ in the inner integral below, if $0 < q \leq m = \dim E$ we have:
\begin{align}
 \nonumber & \int_{G_{n,k}} |P_F K|^{-q} \sigma_{n,k}(dF) = \int_{G_{n,m}} \int_{G_{E,k}} |P_F P_E K|^{-q} \sigma_{E,k}(dF) \sigma_{n,m}(dE) \\
\label{eq:I-move-2} & \leq  c_1 \int_{G_{n,m}} |P_E K|^{-q \frac{k}{m}} \sigma_{n,m}(dE) .
\end{align}
Taking the $-q$-th root, and then the $k$-th root, we obtain:
\[
\I^{1/k}_{k,-q}(K)   \geq c_2 \I^{1/m}_{m,-q\frac{k}{m}}(K)  ,
\] 
for some constants $c_1,c_2 > 0$ independent of $K$, for which equality holds when $K$ is a Euclidean ball; it follows that necessarily $c_2=1$. 
Similarly, if $q < 0$, we obtain the first inequality only reversed, and after taking the $-q$-th root, it remains in the correct direction for us. The case $q=0$ follows similarly.

Finally, if $k < m$ and we have equality in (\ref{eq:I-move}), or equivalently in (\ref{eq:I-move-2}), it follows that $\I^E_{k,-q}(P_E K) = 1$ for $\sigma_{n,m}$-almost-every $E \in G_{n,m}$, and hence for all $E$ by continuity (here $\I^E_{k,-q}$ denotes the functional $\I_{k,-q}$ acting on convex bodies in $E$ instead of $\R^n$).
 Recalling the equality cases in (\ref{eq:Igeq1}), we deduce when $q < m$ ($q=m$) that $P_E K$ is a Euclidean ball (ellipsoid) in $E$ for all $E \in G_{n,m}$. If $m=n$ this concludes the proof; otherwise, $1 \leq k < m \leq n-1$, and it follows that $K$ itself is a Euclidean ball (ellipsoid) by \cite[Corollary 3.1.6 (Theorem 3.1.7)]{GardnerGeometricTomography2ndEd}. 
\end{proof}

\begin{proof}[Proof of Theorem \ref{thm:half-conj}]
For any $0 \leq q \leq m$ and $1 \leq k \leq m \leq n$,
apply Proposition \ref{prop:bootstrap} followed by Jensen's inequality:
\[
\I^{1/k}_{k,-q}(K) \geq \I^{1/m}_{m,-q \frac{k}{m}}(K) \geq \I^{1/m}_{m,-q}(K) .
\]
Equality between the first and last terms implies equality throughout. If $k < m$, equality in the first inequality occurs by Proposition \ref{prop:bootstrap} when $q < m$ ($q=m$) if and only if $K$ is a Euclidean ball (ellipsoid). Equality in the second (Jensen's) inequality occurs when $q > 0$ if and only if $G_{n,m} \ni E \mapsto |P_E K|$ is constant, which for an ellipsoid $K$ and $m \leq n-1$ implies that it must be a Euclidean ball; when $q=0$ we cannot have $q=m$ and so $K$ is already known to be a ball. \end{proof}

\begin{rem} \label{rem:Petty-stuck}
Observe that when $m=k+1$, the above argument only relies on Petty's projection inequality. Consequently, Petty's inequality in combination with Jensen's inequality directly yield for any $k = 2,\ldots,n-2$:
\[
\I^{1/k}_{k,-(k+1)}(K) \geq \I^{1/(k+1)}_{k+1, -k}(K) \geq \I^{1/(k+1)}_{k+1, -(k+2)}(K) . 
\]
Continuing the chain of inequalities, one obtains:
\[
\I^{1/k}_{k,-(k+1)}(K) \geq  \I^{1/(n-1)}_{n-1, -n}(K) ,
\]
which by a final application of Petty's projection inequality is bounded below by $1$. We conclude (without relying on Theorem \ref{thm:main} and only using Petty's inequality) that:
\begin{equation} \label{eq:wrong-moment}
\I_{k,-(k+1)}(K) \geq 1  \;\;\;  \forall k=2,\ldots , n-2 . 
\end{equation}
It does not seem possible to improve the $-(k+1)$-moment in (\ref{eq:wrong-moment}) to the optimal one $-n$ from Conjecture \ref{conj:Lutwak} 
by using similar bootstrap arguments. 
\end{rem}

\section{Concluding remarks} \label{sec:conclude}

\subsection{Extension to compact sets} \label{subsec:compact}

In \cite[Problem 9.4]{GardnerGeometricTomography2ndEd}, Gardner asks whether it would be possible to extend Lutwak's Conjecture \ref{conj:Lutwak} 
to arbitrary compact sets. Certainly, our proof of Proposition \ref{prop:key} (in both Sections \ref{sec:main} and \ref{sec:convexity}) employed convexity in an essential way, and it is not hard to show that the main claims there are simply false for general compact sets. However, the end result of Theorem \ref{thm:main} may very well be valid for general compact sets (as it is hard to imagine a non-convex set which would be more efficient than a Euclidean ball). We briefly provide several remarks in this direction. 

\medskip

First, note that Lemma \ref{lem:convex} only requires for each $K^w$ to be connected. With a little more work, we can thus establish Theorem \ref{thm:Steiner} for a fixed direction $u \in \S^{n-1}$, for compact sets $K$ so that for each $E \in G_{u^{\perp},k-1}$, every section of $K$ parallel to $E^{\perp}$ is connected. However, this property will be destroyed after applying Steiner symmetrization in a consecutive sequence of directions, and so it is not clear how to exploit this to obtain the end result of Theorem \ref{thm:main}. 

\medskip

Second, observe that the validity of the inequality (\ref{eq:main}) of Conjecture \ref{conj:Lutwak} immediately extends to compact sets $K$ whose $k$-dimensional projections are all convex. Indeed, simply apply Theorem~\ref{thm:main} to $\conv(K)$, the convex-hull of $K$, which can only increase the volume of $K$ while preserving the volumes of all $k$-dimensional projections. Equality occurs if and only if $\conv(K)$ is an ellipsoid and $|K| = |\conv(K)|$. 
In particular, this already shows that (\ref{eq:main}) for $k=1$ remains valid for all \emph{connected} compact sets $K$. However, we also observe that:
\begin{thm}[Blaschke--Santal\'o for compact sets]
The inequality (\ref{eq:main}) of Conjecture \ref{conj:Lutwak} for $k=1$ remains valid for arbitrary compact sets $K$, namely $\Phi_1(K) \geq \Phi_1(B_K)$. 
\end{thm}
This provides an interpretation of the Blaschke--Santal\'o inequality which is translation invariant, and which remains valid without any assumptions on convexity nor choice of an appropriate center for $K$ (see also \cite[Corollary 6.4]{LYZ-MomentEntropyInqs}, as well as \cite{Lutwak-ExtendedAffineSurfaceArea,Lehec-FunctionalSantalo} for a version for star-bodies with barycenter at the origin). 
\begin{proof}
Given the compact set $K$, define the following convex body:
\[
\tilde K := \cap_{\theta \in \S^{n-1}} \set{ x \in \R^n \; ; \; \abs{\scalar{x,\theta}} \leq \abs{P_\theta K}/2} . 
\]
We claim that:
\begin{equation} \label{eq:vol-increases}
|K| \leq |\tilde K| . 
\end{equation}
To see this, recall that the Rogers--Brascamp--Lieb--Lutinger inequality \cite{Rogers-BLL,BrascampLiebLuttinger}
 asserts that:
\[
\int_{\R^n} \Pi_{i=1}^m f_i(\scalar{x,\theta_i}) dx \leq \int_{\R^n} \Pi_{i=1}^m f^*_i(\scalar{x,\theta_i}) dx ,
\]
for any measurable functions $f_i : \R \rightarrow \R_+$ and directions $\theta_i \in \S^{n-1}$. Here $f_i^*$ denote the symmetric decreasing rearrangement of $f_i$ (see \cite{BrascampLiebLuttinger} for more details). Applying this to $f_i := 1_{\cdot \theta_i \in P_{\theta_i K}}$ (for which $f_i^* = 1_{[-|P_{\theta_i} K|/2,|P_{\theta_i} K|/2]}$), we obtain:
\[
|K| \leq |\cap_{i=1}^{m} \set{ x \in \R^n \; ; \; P_{\theta_i} x \in P_{\theta_i} K }| \leq |\cap_{i=1}^{m} \set{ x \in \R^n\; ; \; \abs{\scalar{x,\theta_i}} \leq \abs{P_{\theta_i} K}/2}|  .
\]
Using an increasing set of directions $\set{\theta_i}_{i=1}^m$ which becomes dense in $\S^{n-1}$, (\ref{eq:vol-increases})  easily follows. 

On the other hand, we clearly have  $|P_{\theta} \tilde K| \leq |P_\theta K|$ for all $\theta \in \S^{n-1}$, and so:
\begin{equation} \label{eq:BLL}
\Phi_1(K) \geq \Phi_1(\tilde K) \geq \Phi_1(B_{\tilde K}) \geq \Phi_1(B_K) . 
\end{equation}
\end{proof}
\begin{rem}
If there is equality between the left and right terms in (\ref{eq:BLL}), we must have equality in all three inequalities. By Theorem \ref{thm:main}, the second equality implies that the (origin-symmetric) convex body $\tilde K$ is an ellipsoid. The third equality implies that $|\tilde K| = |K|$. Utilizing the first equality is non-trivial, see \cite{Christ-EqualityInBLL}, and so we leave the analysis of equality for another occasion. 
\end{rem}
\begin{rem}
The above argument for $k=1$ does not extend to general $k > 1$. The reason is that the analogue of the Rogers--Brascamp--Lieb--Luttinger inequality for projections onto dimension larger than one is false without some type of separability conditions on the projections (as in \cite[Theorem 3.4]{BrascampLiebLuttinger}). This may be seen, for instance, by the sharpness of the Loomis--Whitney inequality (\ref{eq:LW})  on cubes (as opposed to intersection of spherical cylinders). 
\end{rem}

In summary, we presently do not know how to extend Theorem \ref{thm:main} for $k > 1$ to general compact sets. In particular, this applies to the classical case $k=n-1$. A relaxed variant of such a generalization is to replace the projection volume $|P_F K|$ by its integral-geometric version $\int_{F} \chi(K \cap (f + F^{\perp})) df$, where $\chi$ denotes the Euler characteristic; the two versions coincide when $K$ is convex, but the latter may in general be larger than the former. For this relaxed variant, an extension of the case $k=n-1$ to compact sets $K$ with $C^1$ smooth boundary was obtained by Zhang in \cite{Zhang-AffineSobolevInq}. 

\smallskip 

Finally, we remark that the averaged Loomis--Whitney inequality $\Q_{k,0}(K) \geq \Q_{k,0}(B_K)$ from Subsection \ref{subsec:LW} does extend to general compact sets $K$, but this requires a totally different argument than the one presented here and will be verified elsewhere.

\subsection{Simple new proof of Petty's projection inequality} \label{subsec:Petty}

Our approach in this work suggests that the inequality $\Phi_k(K) \geq \Phi_k(B_K)$ should be interpreted as a generalization of the Blaschke--Santal\'o inequality, corresponding to the case $k=1$. It is a priori equally likely that it could be derived by generalizing Petty's projection inequality, corresponding to the other extremal case $k=n-1$. In fact, we have spent a lot of time trying to derive it ``from the Petty side", without success. Our numerous attempts (see e.g. Remark \ref{rem:Petty-stuck}) all ended up with the inequality $\Q_{k,-(k+1)}(K) \geq \Q_{k,-(k+1)}(B_K)$, having the wrong power $-(k+1)$ instead of the conjectured optimal $-n$. It would be interesting to give an alternative proof of Conjecture \ref{conj:Lutwak} for any $k \in \{1,\ldots,n-2\}$ ``from the Petty side". 

However, one useful byproduct of our failed attempts was the discovery of a new proof of Petty's projection inequality, which is arguably the simplest proof we know. In particular, it completely avoids using the Busemann--Petty centroid inequality \cite[Corollary 9.2.7]{GardnerGeometricTomography2ndEd}. Moreover, it seems to be a ``dual version" of the Meyer--Pajor proof of the Blaschke--Santal\'o inequality \cite{MeyerPajor-Santalo}. We conclude this work by describing it. 

\medskip

Given a convex body $K$, recall the definition of the polar projection body $\Pi^* K$, whose associated norm is given by:
\[
 \norm{\theta}_{\Pi^* K} = |P_{\theta^{\perp}} K| = n V(K,n-1 ; [0,\theta]) \;\;, \;\; \theta \in \S^{n-1} ,
 \]
where  $[0,x]$ denotes the segment between the origin and $x$ (e.g. by (\ref{eq:Fedotov})). 
 By homogeneity, equality between the first and last terms above continues to hold for all $\theta \in \R^{n}$. 
Integration in polar-coordinates immediately verifies that:
\[
|\Pi^* K| = \frac{1}{n} \int_{\S^{n-1}} |P_{\theta^{\perp}} K|^{-n} d\theta ,
\]
and so our goal is to show that:
\begin{equation} \label{eq:simple-Petty}
|\Pi^* K| \leq | \Pi^* S_u K| 
\end{equation}
(originally established by Lutwak--Yang--Zhang \cite{LYZ-Lp-PettyProjection,LYZ-OrliczProjectionBodies}). We will in fact show a much stronger claim, from which (\ref{eq:simple-Petty}) immediately follows after integrating over $u^{\perp}$:

\begin{prop}
For all $y \in u^{\perp}$, $|(\Pi^* K)(y)| \leq |(\Pi^* S_u K)(y)|$, where $L(y) = \{ s \in \R \; ; \; y + s u \in L\}$ is the one-dimensional section of $L$ parallel to $u$ at $y$. 
\end{prop}
\begin{proof}
Fix $y \in u^{\perp}$ and calculate:
\begin{equation} \label{eq:Petty-punch}
|(\Pi^* K)(y)| = \int_{\R} 1_{\norm{y + s u}_{\Pi^* K} \leq 1} ds = \int_{\R} 1_{ V(K,n-1 ; [0,y+ s u]) \leq \frac{1}{n}} ds .
\end{equation}
Consider the linear reflection shadow system $\{K_u(t)\}$ from Lemma \ref{lem:reflection-shadow}. It easily follows from Shephard's paper \cite{Shephard-ShadowSystems} that the function:
\[
\R^2 \ni (s,t) \mapsto f(s,t) := V(K_u(t) , n-1 ; [0, y + su]) \text{ is jointly convex} 
\]
(as the projections of $K_u(t)$ and $[0,y+su]$ onto $u^{\perp}$ do not depend on $t,s$). The function $f$ is also even since:
\[
V(K_u(-t) , n-1 ; [0, y - su]) = V(R_u K_u(t) , n-1 ; R_u [0,y+su]) = V(K_u(t) , n-1 ; [0, y + su])  . 
\]
Hence its level set $\{ (s,t) \in \R^2 \; ; \; V(K_u(t) ; n-1 , [0,y + su]) \leq 1/n \}$ is an origin-symmetric convex body, and so its section at $t=1$ has smaller length than the one at $t=0$:
\[
\int_{\R} 1_{ V(K,n-1 ; [0,y+ s u]) \leq \frac{1}{n}} ds \leq \int_{\R} 1_{V(S_u K,n-1 ; [0,y+ s u]) \leq \frac{1}{n}} ds .
\]
Plugging this into (\ref{eq:Petty-punch}) and rolling back, the assertion follows. 
\end{proof}

Note that instead of fixing $s$ (the $u$-height parameter) and integrating over $y$ (perpendicular to $u$) as in \cite{MeyerPajor-Santalo} and Section \ref{sec:main}, we fix $y$ and integrate over $s$. In all cases, the only inequality used in the proof is between two $(n-k)$-dimensional volumes (which may be thought of as the volumes of two $t$-sections of an $(n-k+1)$-dimensional convex body).

\bibliographystyle{plain} 
\bibliography{../../../ConvexBib}

\def\cprime{$'$} \def\textasciitilde{$\sim$}
\begin{thebibliography}{10}

\bibitem{AitchisonPettyRogers}
P.~W. Aitchison, C.~M. Petty, and C.~A. Rogers.
\newblock A convex body with a false centre is an ellipsoid.
\newblock {\em Mathematika}, 18:50--59, 1971.

\bibitem{Alexandrov-II}
A.~D. Alexandrov.
\newblock Zur {T}heorie gemischter {V}olumina konvexer {K}\"orper; {II}. {N}eue
  {U}ngleichungen zwischen den gemischten {V}olumina und ihre {A}nwendungen.
\newblock {\em Mat. Sb. SSSR}, 2:1205--1238, 1937.

\bibitem{Alexandrov-IV}
A.~D. Alexandrov.
\newblock Zur {T}heorie gemischter {V}olumina konvexer {K}\"orper; {IV}.
  {G}emischte {D}iskriminanten und gemischte {V}olumina.
\newblock {\em Mat. Sb. SSSR}, 3:227--251, 1938.

\bibitem{AGA-Book-I}
S.~Artstein-Avidan, A.~Giannopoulos, and V.~D. Milman.
\newblock {\em Asymptotic geometric analysis. {P}art {I}}, volume 202 of {\em
  Mathematical Surveys and Monographs}.
\newblock American Mathematical Society, Providence, RI, 2015.

\bibitem{Ball-kdim-sections}
K.~Ball.
\newblock Logarithmically concave functions and sections of convex sets in
  $\mathbb{R}^n$.
\newblock {\em Studia Math.}, 88(1):69--84, 1988.

\bibitem{Ball-Shadows}
K.~Ball.
\newblock Shadows of convex bodies.
\newblock {\em Trans. Amer. Math. Soc.}, 327(2):891--901, 1991.

\bibitem{Bertrand-Ellipsoids}
J.~Bertrand.
\newblock D\'emonstration d'un th\'eoreme de g\'eom\'etrie.
\newblock {\em J. Math. Pures Appl.}, 7:215--216, 1842.

\bibitem{Blaschke-Book}
W.~Blaschke.
\newblock {\em Vorlesungen \"uber Differentialgeometrie}, volume~II.
\newblock Berlin-Heidelberg-New York, 1923.

\bibitem{BollobasThomason-Cubes}
B.~Bollob\'{a}s and A.~Thomason.
\newblock Projections of bodies and hereditary properties of hypergraphs.
\newblock {\em Bull. London Math. Soc.}, 27(5):417--424, 1995.

\bibitem{BonnesenFenchelBook}
T.~Bonnesen and W.~Fenchel.
\newblock {\em Theory of convex bodies}.
\newblock BCS Associates, Moscow, ID, 1987.
\newblock Translated from the German and edited by L. Boron, C. Christenson and
  B. Smith.

\bibitem{BLM-SteinerSymmetrizations}
J.~Bourgain, J.~Lindenstrauss, and V.~Milman.
\newblock Estimates related to {S}teiner symmetrizations.
\newblock In {\em Geometric aspects of functional analysis (1987--88)}, volume
  1376 of {\em Lecture Notes in Math.}, pages 264--273. Springer, Berlin, 1989.

\bibitem{BourgainMilman-ReverseSantalo}
J.~Bourgain and V.~D. Milman.
\newblock New volume ratio properties for convex symmetric bodies in {${\bf
  R}^n$}.
\newblock {\em Invent. Math.}, 88(2):319--340, 1987.

\bibitem{BrascampLiebLuttinger}
H.~J. Brascamp, E.~H. Lieb, and J.~M. Luttinger.
\newblock A general rearrangement inequality for multiple integrals.
\newblock {\em J. Functional Analysis}, 17:227--237, 1974.

\bibitem{Brunn-Habilitation}
H.~Brunn.
\newblock \"{U}ber {K}urven ohne {W}endepunkte.
\newblock Habilitationsschrift, Ackermann, M\"{u}nchen, 1889.

\bibitem{BuragoZalgallerBook}
Yu.~D. Burago and V.~A. Zalgaller.
\newblock {\em Geometric inequalities}, volume 285 of {\em Grundlehren der
  Mathematischen Wissenschaften [Fundamental Principles of Mathematical
  Sciences]}.
\newblock Springer-Verlag, Berlin, 1988.

\bibitem{BurtonMani}
G.~R. Burton and P.~Mani.
\newblock A characterisation of the ellipsoid in terms of concurrent sections.
\newblock {\em Comment. Math. Helv.}, 53(4):485--507, 1978.

\bibitem{BusemannStraus}
H.~Busemann and E.~G. Straus.
\newblock Area and normality.
\newblock {\em Pacific J. Math.}, 10:35--72, 1960.

\bibitem{CampiGronchi-LpBusemannPettyCentroid}
S.~Campi and P.~Gronchi.
\newblock The {$L^p$}-{B}usemann-{P}etty centroid inequality.
\newblock {\em Adv. Math.}, 167(1):128--141, 2002.

\bibitem{CampiGronchi-VolumeProductInqs}
S.~Campi and P.~Gronchi.
\newblock On volume product inequalities for convex sets.
\newblock {\em Proc. Amer. Math. Soc.}, 134(8):2393--2402, 2006.

\bibitem{Greeks-AffineQuermassForRandomPolytopes}
G.~Chasapis and N.~Skarmogiannis.
\newblock Affine quermassintegrals of random polytopes.
\newblock {\em J. Math. Anal. Appl.}, 479(1):546--568, 2019.

\bibitem{Christ-EqualityInBLL}
M.~Christ.
\newblock Equality in {B}rascamp--{L}ieb--{L}uttinger inequalities.
\newblock Manuscript, arXiv:1706.02778v1, 2017.

\bibitem{DafnisPaouris-AffineQuermassintegrals}
N.~Dafnis and G.~Paouris.
\newblock Estimates for the affine and dual affine quermassintegrals of convex
  bodies.
\newblock {\em Illinois J. Math.}, 56(4):1005--1021, 2012.

\bibitem{DannPaourisPivovarov-FlagManifolds}
S.~Dann, G.~Paouris, and P.~Pivovarov.
\newblock Affine isoperimetric inequalities on flag manifolds.
\newblock Manuscript, arXiv:1902.09076, 2019.

\bibitem{Favard1933}
J.~Favard.
\newblock Sur les corps convexes.
\newblock {\em J. Math. Pures Appl.}, 12:219--282, 1933.

\bibitem{Fenchel1936}
W.~Fenchel.
\newblock In\'egalit\'es quadratiques entre les volumes mixtes des corps
  convexes.
\newblock {\em C. R. Acad. Sci. Paris}, 203:647--650, 1936.

\bibitem{MahlerIn3D-Simplified}
M.~Fradelizi, A.~Hubard, M.~Meyer, E.~Rold'an-Pensado, and A.~Zvavitch.
\newblock Equipartitions and {M}ahler volumes of symmetric convex bodies.
\newblock arXiv:1904.10765, to appear in Amer. J. Math., 2019.

\bibitem{GardnerSurveyInBAMS}
R.~J. Gardner.
\newblock The {B}runn-{M}inkowski inequality.
\newblock {\em Bull. Amer. Math. Soc. (N.S.)}, 39(3):355--405, 2002.

\bibitem{GardnerGeometricTomography2ndEd}
R.~J. Gardner.
\newblock {\em Geometric tomography}, volume~58 of {\em Encyclopedia of
  Mathematics and its Applications}.
\newblock Cambridge University Press, Cambridge, second edition, 2006.

\bibitem{Gardner-DualAffineQuermassintegrals}
R.~J. Gardner.
\newblock The dual {B}runn-{M}inkowski theory for bounded {B}orel sets: dual
  affine quermassintegrals and inequalities.
\newblock {\em Adv. Math.}, 216(1):358--386, 2007.

\bibitem{GPV-ReverseSantalo}
A.~Giannopoulos, G.~Paouris, and B.-H. Vritsiou.
\newblock The isotropic position and the reverse {S}antal\'{o} inequality.
\newblock {\em Israel J. Math.}, 203(1):1--22, 2014.

\bibitem{Grinberg-Affine-Invariant}
E.~L. Grinberg.
\newblock Isoperimetric inequalities and identities for $k$-dimensional
  cross-sections of convex bodies.
\newblock {\em Math. Ann.}, 291:75--86, 1991.

\bibitem{Gruber-ConvexAndDiscreteGeometry}
P.~M. Gruber.
\newblock {\em Convex and discrete geometry}, volume 336 of {\em Grundlehren
  der Mathematischen Wissenschaften [Fundamental Principles of Mathematical
  Sciences]}.
\newblock Springer, Berlin, 2007.

\bibitem{GruberHobinger}
P.~M. Gruber and J.~H\"{o}binger.
\newblock Kennzeichnungen von {E}llipsoiden mit {A}nwendungen.
\newblock In {\em Jahrbuch \"{U}berblicke {M}athematik, 1976}, pages 9--29.
  1976.

\bibitem{Hadwiger-Book}
H.~Hadwiger.
\newblock {\em Vorlesungen \"{u}ber {I}nhalt, {O}berfl\"{a}che und
  {I}soperimetrie}.
\newblock Springer-Verlag, Berlin-G\"{o}ttingen-Heidelberg, 1957.

\bibitem{IriyehShibata-SymmetricMahlerIn3D}
H.~Iriyeh and M.~Shibata.
\newblock Symmetric {M}ahler's conjecture for the volume product in the
  {$3$}-dimensional case.
\newblock {\em Duke Math. J.}, 169(6):1077--1134, 2020.

\bibitem{Kuperberg-Mahler}
G.~Kuperberg.
\newblock From the {M}ahler conjecture to {G}auss linking integrals.
\newblock {\em Geom. Funct. Anal.}, 18(3):870--892, 2008.

\bibitem{Lehec-FunctionalSantalo}
J.~Lehec.
\newblock A direct proof of the functional {S}antal\'{o} inequality.
\newblock {\em C. R. Math. Acad. Sci. Paris}, 347(1-2):55--58, 2009.

\bibitem{LoomisWhitney}
L.~H. Loomis and H.~Whitney.
\newblock An inequality related to the isoperimetric inequality.
\newblock {\em Bull. Amer. Math. Soc}, 55:961--962, 1949.

\bibitem{Lutwak-Isepiphanic}
E.~Lutwak.
\newblock A general isepiphanic inequality.
\newblock {\em Proc. Amer. Math. Soc.}, 90(3):415--421, 1984.

\bibitem{Lutwak-SomeAffine}
E.~Lutwak.
\newblock On some affine isoperimetric inequalities.
\newblock {\em J. Differential Geom.}, 23(1):1--13, 1986.

\bibitem{Lutwak-Harmonic}
E.~Lutwak.
\newblock Inequalities for {H}adwiger's harmonic {Q}uermassintegrals.
\newblock {\em Math. Ann.}, 280(1):165--175, 1988.

\bibitem{Lutwak-ExtendedAffineSurfaceArea}
E.~Lutwak.
\newblock Extended affine surface area.
\newblock {\em Adv. Math.}, 85(1):39--68, 1991.

\bibitem{Lutwak-Selected}
E.~Lutwak.
\newblock Selected affine isoperimetric inequalities.
\newblock In {\em Handbook of convex geometry, {V}ol. {A}, {B}}, pages
  151--176. North-Holland, Amsterdam, 1993.

\bibitem{LYZ-Lp-PettyProjection}
E.~Lutwak, D.~Yang, and G.~Zhang.
\newblock {$L_p$} affine isoperimetric inequalities.
\newblock {\em J. Differential Geom.}, 56(1):111--132, 2000.

\bibitem{LYZ-MomentEntropyInqs}
E.~Lutwak, D.~Yang, and G.~Zhang.
\newblock Moment-entropy inequalities.
\newblock {\em Ann. Probab.}, 32(1B):757--774, 2004.

\bibitem{LYZ-OrliczProjectionBodies}
E.~Lutwak, D.~Yang, and G.~Zhang.
\newblock Orlicz projection bodies.
\newblock {\em Adv. Math.}, 223(1):220--242, 2010.

\bibitem{LutwakZhang-IntroduceLqCentroidBodies}
E.~Lutwak and G.~Zhang.
\newblock Blaschke-{S}antal\'o inequalities.
\newblock {\em J. Differential Geom.}, 47(1):1--16, 1997.

\bibitem{Mahler-ProofIn2D}
K.~Mahler.
\newblock Ein minimalproblem f\"{u}r konvexe polygone.
\newblock {\em Mathematica (Zutphen)}, B 7:118--127, 1939.

\bibitem{Mahler-Conjecture}
K.~Mahler.
\newblock Ein \"{U}bertragungsprinzip f\"{u}r konvexe {K}\"{o}rper.
\newblock {\em \v{C}asopis P\v{e}st. Mat. Fys.}, 68:93--102, 1939.

\bibitem{ConstantWidthBook}
H.~Martini, L.~Montejano, and D.~Oliveros.
\newblock {\em Bodies of constant width}.
\newblock Birkh\"{a}user/Springer, Cham, 2019.
\newblock An introduction to convex geometry with applications.

\bibitem{MeyerPajor-Santalo}
M.~Meyer and A.~Pajor.
\newblock On the {B}laschke-{S}antal\'{o} inequality.
\newblock {\em Arch. Math. (Basel)}, 55(1):82--93, 1990.

\bibitem{MeyerReisner-SantaloViaShadowSystems}
M.~Meyer and S.~Reisner.
\newblock Shadow systems and volumes of polar convex bodies.
\newblock {\em Mathematika}, 53(1):129--148 (2007), 2006.

\bibitem{MeyerReisner-LocalSantalo}
M.~Meyer and S.~Reisner.
\newblock Ellipsoids are the only local maximizers of the volume product.
\newblock {\em Mathematika}, 65(3):500--504, 2019.

\bibitem{Nazarov-Mahler}
F.~Nazarov.
\newblock The {H}\"{o}rmander proof of the {B}ourgain-{M}ilman theorem.
\newblock In {\em Geometric aspects of functional analysis}, volume 2050 of
  {\em Lecture Notes in Math.}, pages 335--343. Springer, Heidelberg, 2012.

\bibitem{Olovjanishnikov}
S.P. Olovjanishnikov.
\newblock On a characterization of the ellipsoid.
\newblock {\em U\v{c}en. Zap. Leningrad. State Univ. Ser. Mat.}, 83:114--128,
  1941.

\bibitem{PaourisPivovarov-SmallBall}
G.~Paouris and P.~Pivovarov.
\newblock Small-ball probabilities for the volume of random convex sets.
\newblock {\em Discrete Comput. Geom.}, 49(3):601--646, 2013.

\bibitem{Petty-IsoperimetricProblems}
C.~M. Petty.
\newblock Isoperimetric problems.
\newblock In {\em Proceedings of the {C}onference on {C}onvexity and
  {C}ombinatorial {G}eometry ({U}niv. {O}klahoma, {N}orman, {O}kla., 1971)},
  pages 26--41, 1971.

\bibitem{Petty-EllipsoidsSurvey}
C.~M. Petty.
\newblock Ellipsoids.
\newblock In {\em Convexity and its applications}, pages 264--276.
  Birkh\"{a}user, Basel, 1983.

\bibitem{Petty-AffineIsoperimetricProblems}
C.~M. Petty.
\newblock Affine isoperimetric problems.
\newblock In {\em Discrete geometry and convexity ({N}ew {Y}ork, 1982)}, volume
  440 of {\em Ann. New York Acad. Sci.}, pages 113--127. New York Acad. Sci.,
  New York, 1985.

\bibitem{Rogers-BLL}
C.~A. Rogers.
\newblock A single integral inequality.
\newblock {\em J. London Math. Soc.}, 32:102--108, 1957.

\bibitem{RogersShephard-ShadowSystems}
C.~A. Rogers and G.~C. Shephard.
\newblock Some extremal problems for convex bodies.
\newblock {\em Mathematika}, 5:93--102, 1958.

\bibitem{SaintRaymond-Santalo}
J.~Saint-Raymond.
\newblock Sur le volume des corps convexes sym\'{e}triques.
\newblock In {\em Initiation {S}eminar on {A}nalysis: {G}. {C}hoquet-{M}.
  {R}ogalski-{J}. {S}aint-{R}aymond, 20th {Y}ear: 1980/1981}, volume~46 of {\em
  Publ. Math. Univ. Pierre et Marie Curie}. Univ. Paris VI, Paris, 1981.
\newblock Exp. No. 11, 25 pages.

\bibitem{Santalo-BS-paper}
L.~A. Santal\'{o}.
\newblock An affine invariant for convex bodies of {$n$}-dimensional space.
\newblock {\em Portugal. Math.}, 8:155--161, 1949.

\bibitem{Schneider-Book-2ndEd}
R.~Schneider.
\newblock {\em Convex bodies: the {B}runn-{M}inkowski theory}, volume 151 of
  {\em Encyclopedia of Mathematics and its Applications}.
\newblock Cambridge University Press, Cambridge, second expanded edition, 2014.

\bibitem{SchneiderWeil-Book}
R.~Schneider and W.~Weil.
\newblock {\em Stochastic and integral geometry}.
\newblock Probability and its Applications (New York). Springer-Verlag, Berlin,
  2008.

\bibitem{Shephard-ShadowSystems}
G.~C. Shephard.
\newblock Shadow systems of convex sets.
\newblock {\em Israel J. Math.}, 2:229--236, 1964.

\bibitem{Soltan-EllipsoidsSurvey}
V.~Soltan.
\newblock Characteristic properties of ellipsoids and convex quadrics.
\newblock {\em Aequationes Math.}, 93(2):371--413, 2019.

\bibitem{Urysohn-Inq}
P.~Urysohn.
\newblock Mean width and volume of convex bodies in $n$-dimensional space.
\newblock {\em Rec. Math. Soc. Math. Moscow}, 31:477--486, 1924.
\newblock (In Russian).

\bibitem{Zhang-ReversePetty}
G.~Zhang.
\newblock Restricted chord projection and affine inequalities.
\newblock {\em Geom. Dedicata}, 39(2):213--222, 1991.

\bibitem{Zhang-AffineSobolevInq}
G.~Zhang.
\newblock The affine {S}obolev inequality.
\newblock {\em J. Differential Geom.}, 53(1):183--202, 1999.

\bibitem{ZouXiong}
D.~Zou and G.~Xiong.
\newblock New affine inequalities and projection mean ellipsoids.
\newblock {\em Calc. Var. Partial Differential Equations}, 58(2):Paper No. 44,
  18, 2019.

\end{thebibliography}

\end{document}